\documentclass[reqno,12pt]{amsart}
\newcommand{\be}{\begin{eqnarray}}
	\newcommand{\ee}{\end{eqnarray}}
\newcommand{\beq}{\begin{equation}}
	\newcommand{\eeq}{\end{equation}}
\newcommand{\ben}{\begin{eqnarray*}}
	\newcommand{\een}{\end{eqnarray*}}
\newcommand{\dis}{\displaystyle}
\newtheorem{theorem}{Theorem}

\newtheorem{lemma}[theorem]{Lemma}

\newtheorem{proposition}[theorem]{Proposition}

{\catcode`\@=11\global\let\AddToReset=\@addtoreset
	\AddToReset{equation}{section}
	
	\AddToReset{theorem}{section}

	\usepackage{graphicx}
	\usepackage{psfrag}
	\usepackage{color}

	\newcommand{\RR}{\mathbb R}

\begin{document}
\title[Multiple existence of ground states]
      {Multiplicity results for ground state solutions of
            a semilinear equation via abrupt changes in magnitude of the nonlinearity   }\thanks{This research was supported by
        FONDECYT- 1190102 for the first author,
        FONDECYT-1210241  for the second author and FONDECYT-1170665 for the  third author.}
\author{Carmen Cort\'azar}
\address{Departamento de Matem\'atica, Pontificia
        Universidad Cat\'olica de Chile,
        Casilla 306, Correo 22,
        Santiago, Chile.}
\email{\tt ccortaza@mat.puc.cl}
\author{Marta Garc\'{\i}a-Huidobro}
\address{Departamento de Matem\'atica, Pontificia
        Universidad Cat\'olica de Chile,
        Casilla 306, Correo 22,
        Santiago, Chile.}
\email{\tt mgarcia@mat.puc.cl}
\author{Pilar Herreros}
\address{Departamento de Matem\'atica, Pontificia
        Universidad Cat\'olica de Chile,
        Casilla 306, Correo 22,
        Santiago, Chile.}
\email{\tt pherrero@mat.puc.cl}


\begin{abstract}
Given $k\in\mathbb N$, we define a class of continuous piecewise functions  $f$ having abrupt but controlled magnitude changes so that the problem
 $$ \Delta u
+f(u)=0,\quad x\in \RR^N, N> 2, $$
 has at least $k$ radially symmetric ground state solutions.

\end{abstract}

\maketitle

\section{Introduction and main results}

In this paper we define a class of  continuous nonlinearities $f$ so that the problem
\begin{eqnarray}\label{pde}
\begin{gathered}
 \Delta u
+f(u)=0,\quad x\in \RR^N, N> 2, \\
 \lim\limits_{|x|\to\infty}u(x)=0,
\end{gathered}
\end{eqnarray}
 has multiple positive solutions. To this end we consider the radial version of \eqref{pde}, that is
\begin{eqnarray}\label{eq2}
\begin{gathered}
u''+\frac{N-1}{r}u'+f(u)=0,\quad u(r)>0\mbox{ for all } r>0,\quad N> 2, \\
u'(0)=0,\quad \lim\limits_{r\to\infty}u(r)=0,
\end{gathered}
\end{eqnarray}
where all throughout this article  $u'$ denotes differentiation of $u$ with respect to $r$.

Any nonconstant solution to \eqref{pde} is called a bound state solution. Bound state solutions such that $u(x)>0$ for all $x\in\mathbb R^N$ are referred to as a first bound state solution, or  a ground state solution.

The uniqueness problem for positive solutions to problem \eqref{pde} has been extensively studied during the past decades, see for example  \cite{fls, ms, pel-ser2, pu-ser, st}.

The multiplicity problem has been studied for the following non-autonomous problem
$$-\Delta u=f(x,u),\qquad u(x)\to0\quad\mbox{as $|x|\to\infty$}$$
for  $f$ of the form $f(x,u)=g(x,u)-a(x)u$ by \cite{at1, at2, aw,  cmp, cps1, cao, dwy, hl,  wy, mp21}. Under different assumptions on the nonnegative function $g$ and the coefficient $a$, they have established existence of multiple ground state solutions. More recently, Cerami and Molle \cite{cm19}, considered $f(x,u)=u^p-a(x)u-b(x)u^q$, $q<p<\frac{N+2}{N-2}$, introducing the nonzero term $b(x)$ and gave conditions to obtain infinitely many ground states. The autonomous case was studied in \cite{ddg} for $f(u)=-u+u^p+\lambda u^q$ with $N=3$, $1<q<3$, $p$ near $5$, where they prove that there if $\lambda$ is large enough, then there exist at least three radial ground state solutions to this problem. We also mention the work by Wei and Wu, \cite{wei-wu}, where the authors consider the nonlinearity $f(u)=|u|^{2^*-2}u+\lambda u+\mu |u|^{q-2}u$, where $2^*=\frac{2N}{N-2}$ is the well known critical exponent and among other results, they prove that if $N=3$, $2<q<10/3$, under some conditions in $\mu>0$ the problem has a second ground state for some $\lambda<0$.

\bigskip

In \cite{cghh2} we established a multiplicity result for bound states by considering  nonlinearities behaving like $u^p-u$ at the start but  having multiple sign changes so that its primitive $F(s)=\int_0^sf(t)dt$ would have positive local  maxima. In the present  work, we will define $f$ piecewise, starting with a function $f_1$  satisfying assumptions $(H_1)$-$(H_4)$ below so that problem \eqref{eq2} with $f=f_1$ has a unique  ground state solution, and then having abrupt but controlled \lq\lq magnitude\rq\rq\  changes.
Therefore our  results will consider a nonlinearity of the form
$$f=f_1\chi_{[0,\alpha_*+\epsilon_1]}+\sum_{i=2}^kL_{i-1}
\chi_{[\alpha_{i-1},\alpha_{i-1}+\epsilon_{i-1}]}+
\sum_{i=2}^{k-1}A_{i}^2f_{i}\chi_{[\alpha_{i-1}+\epsilon_{i-1},\alpha_{i}]}+
A_{k}^2f_{k}\chi_{[\alpha_{k-1}+\epsilon_{k-1},\gamma)},$$
where $k\in\mathbb N$, $\chi_E$ denotes the characteristic function of the set $E$, $L_{i-1}$ is a linear function defined so that $f$ is continuous, $\epsilon_i$ and $A_i$ are positive constants  and the functions $f_i$, $i\ge 2$, are {\bf any} positive continuous  functions defined in $[\alpha_*,\gamma)$, see Figure \ref{fig1}.

\begin{figure}[h]
  \includegraphics[scale=1]{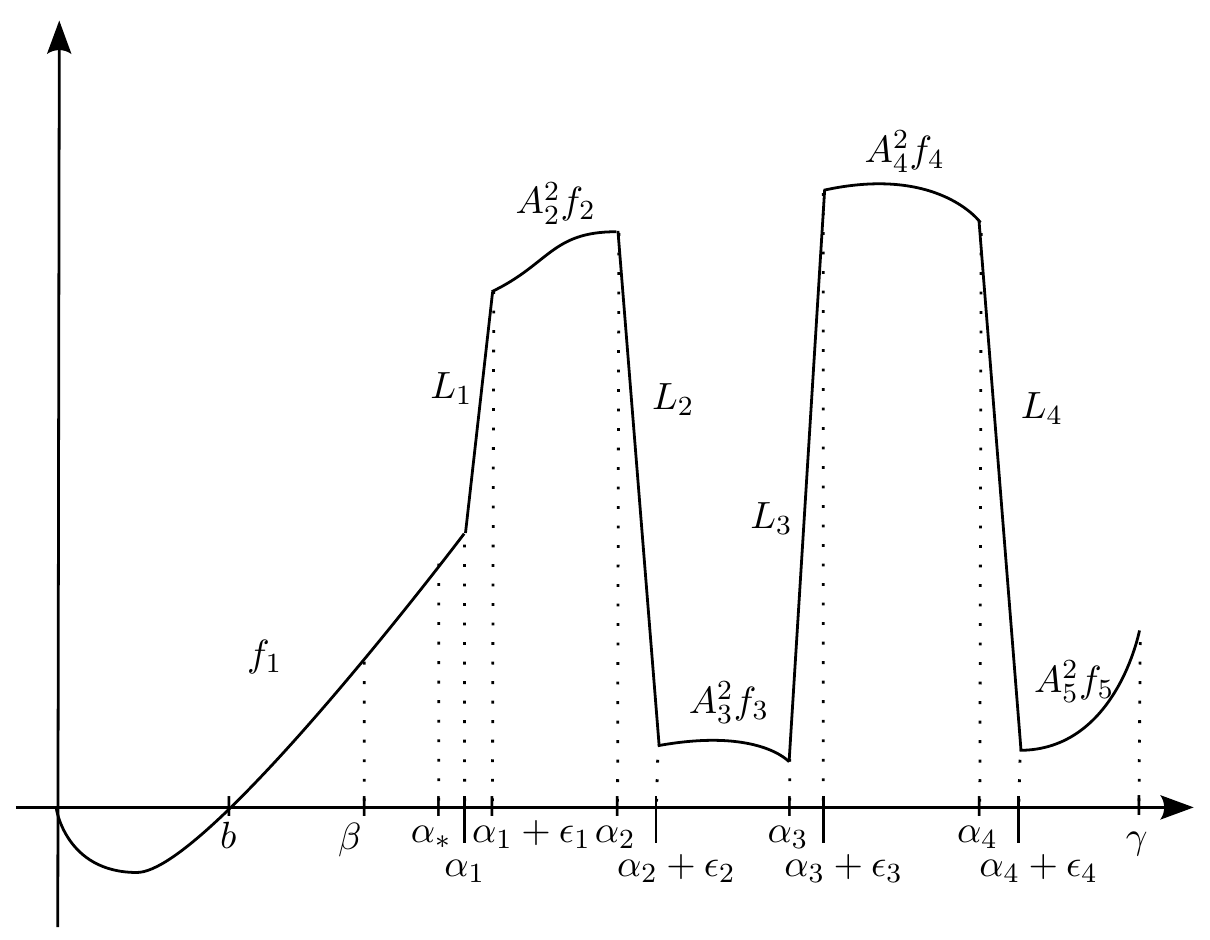}
  \caption{A function $f$ for $k=5$}\label{fig1}
\end{figure}
\medskip

We start with the case $k=2$, so that
 \begin{eqnarray}\label{f}
f(s)=\begin{cases}
f_1(s) &  s\leq \alpha_*+\epsilon_1=\alpha_1\\
L_1(s) & \alpha_1 \leq s\leq \alpha_1+\epsilon_1\\
A_2^2f_2(s) &  s\geq \alpha_1+\epsilon_1,
\end{cases}
\end{eqnarray}
where $L_1(s)$ is the line from $(\alpha_1, f_1(\alpha_1))$ to $(\alpha_1+\epsilon_1, A^2f_2(\alpha_1+\epsilon_1))$ and  $\alpha_*$ is given by $(H_4)$. The constants  $\epsilon_1$ and $A_2$ will be determined.

 The continuity assumption on $f$ is crucial to guarantee continuous dependence of the solutions on initial conditions. We have chosen the transition functions $L_i$ to be linear for simplicity, and it could be avoided at the cost of imposing that $f_2$ be monotone nondecreasing.

We will assume the following conditions on the nonlinearity $f$:

\begin{enumerate}
 \item[$(H_1)$]  $f_1\in C[0,\infty) \cap C^1(0,\infty)$, $f_1(0)=0$ and there exist $b\ge 0$ such that $f_1(s)>0$ for $s>b$,
	      $f_1(s)\le 0$ for $s\in[0,b]$ and moreover
	      $f_1(s)<0$ on $(0,\epsilon)$ for some $\epsilon>0$;
also, by setting
$ F_1(s) = \int_0^s f_1(t) dt,
$
we assume that there exists a unique finite $\beta\ge b$ such that $F(\beta)=0$.
\end{enumerate}

\begin{enumerate}
 \item[$(H_2)$] $(F_1/f_1)'(s) > (N-2)/(2N)$ for all $s>\beta$;
 \item[$(H_3)$] $\displaystyle f_1/(s-b)$ is increasing for all $s>b$.
\item[$(H_4)$] There is an initial condition $\alpha_*$ such that the problem
\begin{eqnarray}\label{pdef1}
\begin{gathered}
u''+\frac{N-1}{r}u'+f_1(u)=0,\quad r>0,\quad N> 2,\\
u(0)=\alpha_*,\quad u'(0)=0,
\end{gathered}
\end{eqnarray}
is a ground state solution.

\item[$(H_5)$] $f_i$ is a positive continuous function defined on $[\alpha_*,\gamma)$ for some $\alpha_*<\gamma\le\infty$ for all $i\ge 2$.

\end{enumerate}

\medskip
Our first result is the following.
\begin{theorem}\label{second solution}
Assume that $f_1$, $\alpha_*$ and $f_2$ satisfy the assumptions above.
Then, there exist positive constants $\bar\epsilon$ and  $\bar A$ such that for any $0<\epsilon_1<\bar\epsilon$ and $A_2>\bar A$, problem \eqref{pde} with $f$ given by \eqref{f} has at least two ground state solutions.
\end{theorem}

Moreover, if $f$ satisfies a subcritical type condition at infinity,  similar to the one introduced first by Castro and Kurepa see \cite{cku} and used by Gazzola, Serrin and Tang in \cite{gst},
\begin{enumerate}
\item[$(H_6)$]
Let $Q(s):=2NF(s)-(N-2)sf(s)$, where $F(s)=\int_0^sf(t)dt$. We assume that $Q$ is bounded from below  in $(0,\infty)$ and that
 there exists $\theta\in(0,1)$
\begin{equation}\label{subcrit}
\lim_{s\to\infty}\left(\inf_{s_1,s_2\in[\theta s,s]}Q(s_2)\Bigl(\frac{s}{f(s_1)}\Bigr)^{N/2}\right)=\infty,
\end{equation}
\end{enumerate}
we can obtain a third solution to \eqref{pde}. We have
\begin{theorem}\label{third solution}
Assume that $f_1$, $\alpha_*$ and $f_2$ be as in Theorem \ref{second solution} with $\gamma=\infty$, and let $\epsilon_1$ and $A_2$ be as in its conclusion. If $f$ satisfies $(H_6)$, then problem \eqref{pde} has at least three ground state solutions.
\end{theorem}
\medskip

We now consider the case $k=3$, that is $f$ given by
 \begin{eqnarray}\label{f_3}
f(s)=\begin{cases}
f_1(s) &  s\leq \alpha_1\\
L_1(s) & \alpha_1 \leq s\leq \alpha_1+\epsilon_1\\
A_2^2f_2(s) &   \alpha_1+\epsilon_1\leq  s\leq \alpha_2\\
L_2(s) & \alpha_2 \leq s\leq \alpha_2+\epsilon_2\\
A_3^2f_3(s) &  s\geq \alpha_2+\epsilon_2,
\end{cases}
\end{eqnarray}
where  $L_1(s)$ is the line from $(\alpha_1, f_1(\alpha_1))$ to $(\alpha_1+\epsilon_1, A_2^2f_2(\alpha_1+\epsilon_1))$,  $L_2(s)$ is the line from $(\alpha_2, A_2^2f_2(\alpha_2))$ to $(\alpha_2+\epsilon_2, A_3^2f_3(\alpha_2+\epsilon_2))$ and  $\alpha_1,\  \epsilon_1$ and  $A_2$ are constants that satisfy Theorem \ref{second solution}. The constants $ \alpha_2,\ \epsilon_2$ and $A_3$  will be determined later.

 \begin{theorem}\label{solution f_3}
Under  assumptions $(H_1)$-$(H_5)$ ,
there exist positive  constants $\epsilon_1$, $\epsilon_2$ ,  $A_2$ and $A_3$ such that  problem \eqref{pde} with $f$ given by \eqref{f_3} has at least three ground state solutions.
\end{theorem}

Finally, we address the general case $k\geq 4$.

 \begin{theorem}\label{multiple}
Let   $f_i$, $i=1,\ldots,k$ satisfy assumptions $(H_1)$-$(H_5)$ For $i=2...k$, there exists constants  $\epsilon_i>0$, $ A_{i}>0$  and $\alpha_i$ with the condition $\alpha_*<\alpha_{i-1}+\epsilon_{i-1}<\alpha_{i}$  such that problem \eqref{pde} with
$$f=f_1\chi_{[0,\alpha_*+\epsilon_1]}+\sum_{i=2}^kL_{i-1}
\chi_{[\alpha_{i-1},\alpha_{i-1}+\epsilon_{i-1}]}+
\sum_{i=2}^{k-1}A_{i}^2f_{i}\chi_{[\alpha_{i-1}+\epsilon_{i-1},\alpha_{i}]}+A_{k}^2f_{k}\chi_{[\alpha_{k-1}+\epsilon_{k-1},\gamma)}$$
has at least $k$ ground state solutions.
\end{theorem}
\noindent{\bf Remark}: It will follow from the proof of Theorem \ref{third solution} that if $k$ is even and $f_k$ satisfies $(H_6)$, we obtain a $(k+1)$th ground state.
\bigskip

Our results are based on the study of  initial value problems of the form
\begin{eqnarray}\label{ivp}
\begin{gathered}
u''+\frac{N-1}{r}u'+f(u)=0,\quad r>0,\quad N> 2,\\
u(r_0)=\alpha,\quad u'(r_0)=\bar\alpha
\end{gathered}
\end{eqnarray}
for some $r_0\ge 0,\ \alpha>0$ and $\bar\alpha\in\mathbb R_0^-$. Under our assumptions on $f$ the solution to this problem is unique. In case that $r_0=0$, we set $\bar\alpha=0$ so it corresponds to a radially symmetric solution of problem \eqref{pde}. In section \ref{prel1} we give some preliminary properties of these solutions.

In section \ref{sec f1} we study the solutions to \eqref{ivp}  when $u(r_0)=\alpha_*$ and $u'(r_0)=\bar\alpha<0$, so that $f=f_1$ in its range. Under our assumptions $(H_1)$ through $(H_4)$ we will use  the well known functional $P$ introduced  by Erbe and Tang in \cite{et}, to compare our solutions to the known positive solutions of  \eqref{ivp}. We prove that, under some conditions on $r_0$ and $\bar \alpha$, these solutions stay positive.

Next, in section  \ref{sec f2}, we will study how the constants $A_2$ and $\epsilon_1$ affect the solutions of \eqref{ivp} with $r_0=0$, $\alpha>\alpha_*$ and $\bar\alpha=0$, proving that for small enough $\epsilon_1$ the effect can be controlled. The interesting part is that for $A_2$ big enough, solutions will reach $\alpha_*$ satisfying the conditions found in the previous section. Thus, there is a positive solution with initial condition $\alpha_1^*>\alpha_*$, which implies the existence of a ground state solution between them and proves Theorem \ref{second solution}.

When $f$ also satisfies $(H_6)$, using results from \cite{gst} we show that solutions with large enough initial value $\alpha$ change sign. This will prove the existence of a third ground state solution with initial condition larger than the one before, proving Theorem \ref{third solution}.

We finish this section adding another magnitude change in $f$, making it small for values of $\alpha$ larger than the $\alpha_1^*$ found in Theorem \ref{second solution}. We prove that if $A_3$ is small enough there will be another ground state solution with initial condition $\alpha_2^*>\alpha_1^*$, proving Theorem \ref{solution f_3}.

In section \ref{fk} we prove the general case, Theorem \ref{multiple}. Alternating between big and small $A_i$ we can use the arguments from the previous theorems to determine recursively the constants  $\alpha_i,\  \epsilon_i$ and  $A_i$ so that the solutions with initial condition $\alpha_i$ are positive if $i$ is even, and change sign if $i$ is odd, hence obtaining $k$ ground state solutions.

We finish the paper with some examples in section \ref{examples}. For the case $k=2$ they show the different behavior of solutions with large $\alpha$, showing that some condition on $f_2$ is necessary for solutions with large $\alpha$ to change sign.

We also include an Appendix, where we sketch the proofs of some of the properties of solutions to the initial value problem \eqref{ivp} given in section \ref{prel1}. These are known facts,  but, to  our knowledge there is no paper where the results are proven with our conditions.

\section{Preliminaries}\label{prel1}

The aim of this section is to establish several properties of the solutions to the initial value problem \eqref{ivp} in the case $r_0=0$ and $\bar\alpha=0$.
with $f\in C[0,\infty)$ and such that $(H_1)$ is satisfied. This problem has a unique solution defined for all $r>0$ for any $\alpha>0$ and we denote it by $u(\cdot,\alpha)$.

It can be seen that for $\alpha\in(b,\infty)$, one has $u(r,\alpha)>0$ and $u'(r,\alpha)<0$ for
$r$ small enough, and thus  we can define
$$R(\alpha):=\sup\{r>0\ |\ u(s,\alpha)>0\mbox{ and }u'(s,\alpha)<0\ \mbox{ for all }s\in(0,r)\}.$$

Following \cite{pel-ser1}, \cite{pel-ser2} we set
\begin{eqnarray*}
{\mathcal N}&=&\{\alpha \in (b,\infty)\ :\ u(R(\alpha),\alpha)=0\quad\mbox{and}\quad u'(R(\alpha),\alpha)<0\}\\
{\mathcal G}&=&\{\alpha\in (b,\infty) \ :\ u(R(\alpha),\alpha)=0\quad\mbox{and}\quad u'(R(\alpha),\alpha)=0\}\\
{\mathcal P}&=&\{\alpha\in (b,\infty) \ :\ u(R(\alpha),\alpha)>0\}.
\end{eqnarray*}

The following propositions state some known facts, but, to  our knowledge there is no paper where the result is proven with our conditions. So, for the sake of completeness,  we will give a sketch of the proof  in the appendix.

\begin{proposition}
\label{open-sets}\mbox{ }\\

i) If $f $ satisfies  $(H_1)$ and $(H_5)$, then the sets ${\mathcal N}$ and ${\mathcal P}$ are open sets.\\

ii) If $f_1 $  satisfies the assumptions $(H_1)$, $(H_2)$, and $(H_4)$ then for the problem
\begin{eqnarray}\label{eq22}
\begin{gathered}
u''+\frac{N-1}{r}u'+f_1(u)=0,\quad r>0,\quad N> 2,\\
u'(0)=0,\quad \lim\limits_{r\to\infty}u(r)=0,
\end{gathered}
\end{eqnarray}
the ground state is unique and
${\mathcal P}$=$(b,\alpha_*)$ and ${\mathcal N}$ =$(\alpha_*, \infty)$.

\end{proposition}

\begin{proof}

i) See, for example,  \cite{pel-ser1}, \cite{pel-ser2},  \cite{cfe1}.

ii)From Theorem \ref{unic}.
\end{proof}
\medskip

 We will use the functional introduced first by Erbe and Tang in \cite{et}, as well as many of their ideas.

For a solution $u$ of \eqref{ivp}, let

$$P(s)=-2N\frac{F}{f}(s)\frac{r^{N-1}(s)}{r'(s)}-\frac{r^N(s)}{(r'(s))^2}
-2r^N(s)F(s),\quad $$
where $r(s)$ denotes the inverse of $u$ in the interval $(0,R(\alpha))$, and note that

\begin{equation}P_s(s)=\frac{\partial P}{\partial s}(s)=\left(N-2-2N\Bigl(\frac{F}{f}\Bigr)'(s)\right)\frac{r^{N-1}(s)}{r'(s)}.
\end{equation}

\begin{proposition}\label{final1}

Let f satisfy $(H_1)$ and $(H_2 )$. Let $u_1$, $u_2$ be two solutions of \eqref{ivp},  such that $u_1(r_I)= u_2(r_I)= s_I$,  $u'_1(r_I)= \bar \alpha_1$, $u'_2(r_I)= \bar \alpha_2$ with $ \bar \alpha_2 <\bar \alpha_1<0 $. Assume furthermore that there exists $\bar s>0$ such that $u_2'<0$ in $(\bar s,s_I]$ and $u_2'(r_2(\bar s))=0$.
If

i)   $s_I >\beta$  and  $ P_2(s_I)< P_1(s_I)< 0 $ , where $P_i=P(s,u_i)$,

or

ii) $s_I \leq \beta$ ,\\
then, there exists $t\in[\bar s,s_I)$ such that $u_1'(r_1(t))=0$ and $r_1(s)>r_2(s)$ in $(t,s_I)$.

\end{proposition}

\begin{proof} From Proposition \ref{finalI}.

\end{proof}

 \begin{proposition}\label{P}

Let f satisfy $(H_1)$-$(H_4 )$ and $u_1$, $u_2$ be two solutions of \eqref{ivp},  such that $u_1(0)= \alpha_1$ $u_2(0)= \alpha_2$   $u'_1(0)= u'_2(0)=0$,  with    $ b<\alpha_1,  \alpha_2< \alpha_*$. Then $u_1$ and $u_2$ intersect each other once and only once in $(0,\min\{R(\alpha_1),R(\alpha_2)\})$. Such intersection occurs at a point  $u_1=u_2>b$.

\end{proposition}

\begin{proof}
From  Proposition \ref{SI1} and then Proposition \ref{finalI}.

\end{proof}

\section{Analysis of the solutions of \eqref{ivp} after crossing the value $\alpha_*$}\label{sec f1}

We begin our analysis by studying the behavior of solutions to \eqref{ivp} after they reach $u(r_0)=\alpha_*$. Since for $u\leq\alpha_*$ the function $f=f_1$, we can compare them with the solution of the better understood case where $f$ satisfies properties $(H_1)$-$(H_4)$.  Our goal is to prove the following proposition.

\begin{proposition}\label{EnP-f1}  Let $f_1$ satisfying the properties $(H_1)$-$(H_4)$,   and let $u$ be solution to
\begin{eqnarray}\label{eq rad}
u''+\frac{N-1}{r}u'+f(u)=0
\end{eqnarray}
 that reaches the value $\alpha_*$ in $(H_4)$ at some $r_*\in(0,R)$.
Then given $\bar a,\ \bar b\in\mathbb R^+$, with  $ \bar b<(\alpha_*-b) (N-2) $, if $R>0$ is sufficiently small and $\bar a\le r_*|u'(r_*)|\le \bar b$, it holds that $u(r)>0$ for all  $r\ge r_*$.

\end{proposition}

We will use the following lemma to prove this proposition.

\begin{lemma}\label{u}
Under the assumptions of Proposition \ref{EnP-f1}, let $\delta:=\min\{\frac{\alpha_*- b}{2} -\frac{\bar b}{2(N-2) },\frac{\bar a}{2(N-2)} \}$, then

 i) If $u(r)\in(\alpha_*-\delta,\alpha_*)$, then $0<r<2R$ and $r^{N-1}|u'(r)|\le CR^{N-2}(1+R^2)$, where $C=C(||f_1||_\infty,\bar b)$.

ii) Let $ a_2 $  such that $(\alpha_*-a_2)= \left(\frac{(\alpha_*-b)}{2}+\frac{\bar b}{2(N-2) }\right) $ and $r_a$ such that $u(r_a)=a_2$ then

$$r_a \geq \sqrt{\left(\frac{(\alpha_*-b)}{2}-\frac{\bar b}{2(N-2) }\right) \frac{2N}{||f_1||_\infty }}:=C_2. $$
Observe that $2\delta< (\alpha_*-a_2) < (\alpha_*-b)$, see Figure \ref{fig-2}.

\end{lemma}

\begin{figure}[h]
  \includegraphics[scale=1]{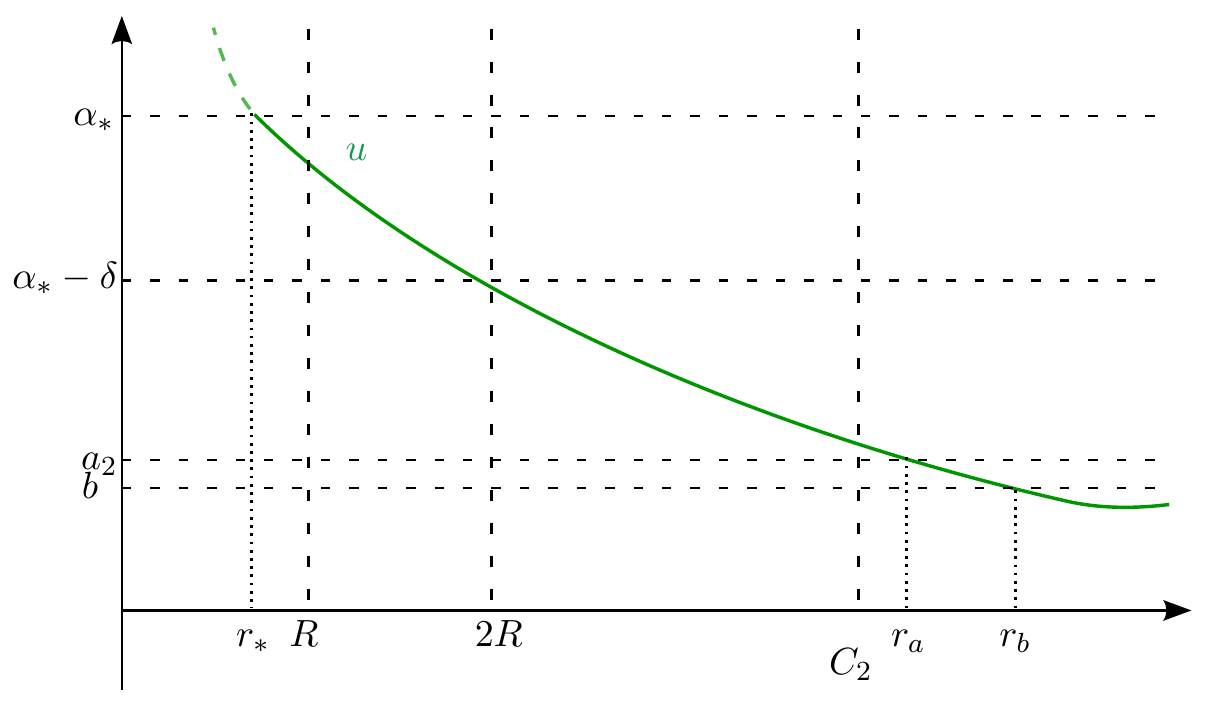}
  \caption{}\label{fig-2}
\end{figure}

\begin{proof}	
	
\noindent  We first prove i). We have
	$$r^{N-1}u'(r)=r_*^{N-1}u'(r_*)-\int_{r_*}^rt^{N-1}f_1(u(t))dt\le
r_*^{N-1}u'(r_*),$$
hence dividing by $r^{N-1}$ and integrating over $(r_*,r)$, we find that
$$u(r)-u(r_*)\le \frac{r_*^{N-1}u'(r_*)}{2-N}(r^{2-N}-r_*^{2-N})$$
implying
$$\delta\ge u(r_*)-u(r)\ge \frac{r_*|u'(r_*)|}{N-2}\left(1-\left(\frac{r_*}{r}\right)^{N-2}\right).$$
From the conditions $\bar a\le r_*|u'(r_*)|$ and the definition of $\delta$ we obtain
$$1-\left(\frac{r_*}{r}\right)^{N-2}\le \frac{1}{2}$$
hence
$$\frac{r_*}{r}\ge \Bigl(\frac{1}{2}\Bigr)^{1/(N-2)}\ge\frac{1}{2}$$
and therefore $r\le 2r_*\le 2R$.

On the other hand, if $u\in[\alpha_*-\delta,\alpha_*]$, then
$$r^{N-1}|u'(r)|\le r_*^{N-1}|u'(r_*)|+\int_{r_*}^rt^{N-1}f_1(u(t))dt\le \bar b R^{N-2}+||f_1||_\infty\frac{(2R)^N}{N}$$
and thus i) follows.

 Proof of ii) As
$$r^{N-1}|u'(r)|\le r_*^{N-1}|u'(r_*)|+\int_{r_*}^rt^{N-1}f_1(u(t))dt$$
$$|u'(r)|\le \left(\frac {r_*}{r}\right)^{N-1}|u'(r_*)|+||f_1||_\infty \frac{r}{N}$$

Integrating over $[r_*,r]$,
$$ (\alpha_* -u(r)) \leq  \frac{r_*|u'(r_*)|}{N-2}\Bigl(1-\left(\frac {r_*}{r}\right)^{N-2}\Bigr) +||f_1||_\infty \frac{r^2}{2N} $$
$$ \leq \frac{\bar b}{N-2}+||f_1||_\infty \frac{r^2}{2N} $$

Hence, for $r=r_a$  $$ r_a   \geq \sqrt{\left(\frac{(\alpha_*-b)}{2}-\frac{\bar b}{2(N-2) }\right) \frac{2N}{||f_1||_\infty }}=C_2.$$

\end{proof}

\subsection*{Proof of Proposition \ref{EnP-f1}}

We will do this proof in several steps

\begin{lemma}(Step 1)\label{vw}
Under the assumptions of Proposition \ref{EnP-f1},  let  $\delta:=\min\{\frac{\alpha_*- b}{2} -\frac{\bar b}{2(N-2) },\frac{\bar a}{2(N-2)} \}$ if $R$ is sufficiently small and $r_*< R$ then there is a solution  $v$  to \eqref{ivp} with $r_0=0$, $\bar\alpha=0$  and $\alpha\in[\alpha_*-\delta/2,\alpha_*-\delta/4]$ such that $u$ intersects $v$ at least twice at values greater than $b$.
\end{lemma}

\begin{proof}		
 Assume that $$2R<\sqrt{\frac{N\delta}{8||f_1||_\infty}}$$ we will first prove that  a first intersection  occurs for any such solution $v$  at some $r_1<2R$. Note that $2R<C_2/2$ where $C_2$ is the constant in lemma \ref{u}. By integration we have that
$$r^{N-1}|v'(r)|=\int_0^rt^{N-1}f_1(v(t))dt\le ||f_1||_\infty\frac{r^N}{N},$$
implying $|v'(r)|\le ||f_1||_\infty\frac{r}{N},$ and integrating once more, $\alpha-v(r)\le ||f_1||_\infty\frac{r^2}{2N}$ and thus when $v(r)=\alpha_*-\delta$, we have
$r^2\ge \frac{\delta N}{||f_1||_\infty}	$ implying
$$r\ge \sqrt{\frac{\delta N}{2||f_1||_\infty}}> 2R.$$
Since by Lemma \ref{u} when $u(r)=\alpha_*-\delta$ one has $r<2R$, the solutions must have intersected at some $r_1<2R<C_2/2$, see $r_v$ or $r_w$ in Figure \ref{fig-3}.

\begin{figure}[h]
	\includegraphics[scale=1]{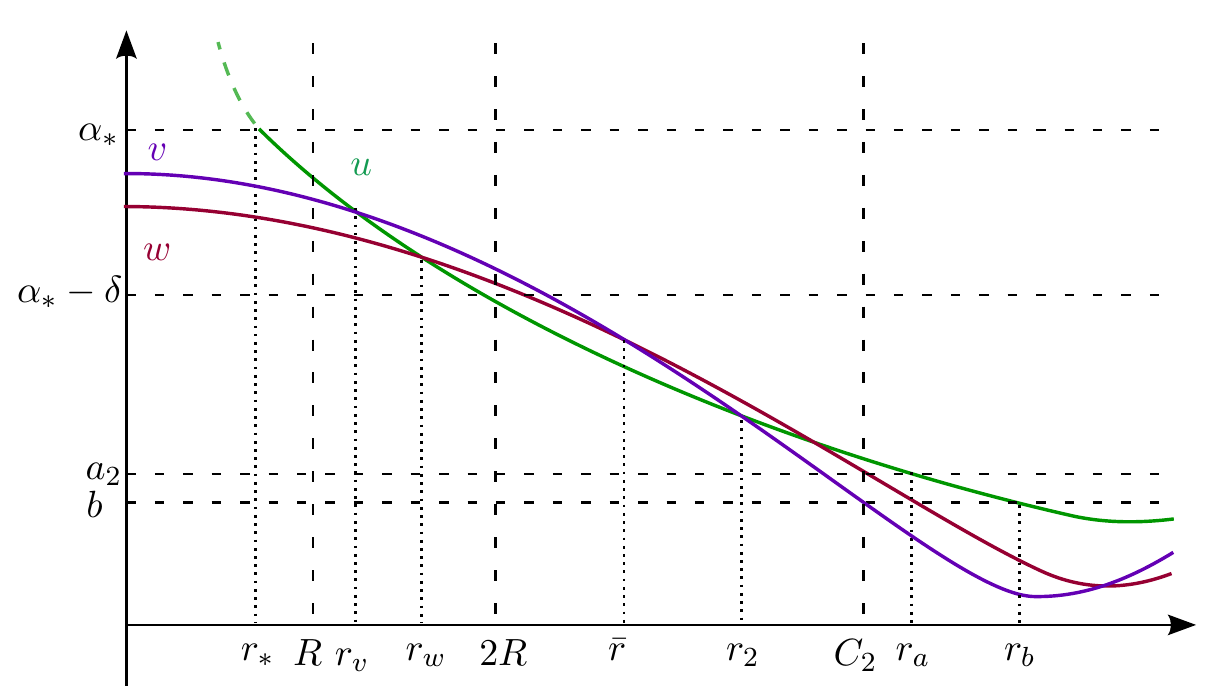}
	\caption{}\label{fig-3}
\end{figure}

We will see next that $u$ must  intersect at least once more some of these solutions while they are greater than $b$. To this end we set $v$ the solution to \eqref{ivp} such that $r_0=0$, $\bar\alpha=0$ and $\alpha=\alpha_*-\delta/4$ and $w$ the solution to \eqref{ivp} such that $r_0=0$, $\bar\alpha=0$ and $\alpha=\alpha_*-\delta/2$.\\
By Lemma \ref{P}, these two solutions, $v$ and $w$ intersect each other once and only once before reaching the value $b$, hence
 $v(r)>w(r)$ for $r\in(0,\bar r)$ and $v(\bar r)=w(\bar r)$ for some $\bar r>0$, see Figure \ref{fig-3}.
Using again the equation and integrating over $(0,\bar r)$ we have
$$r^{N-1}|(v-w)'(r)|=\int_0^rt^{N-1}(f_1(v(t))-f_1(w(t)))dt\le ||f_1||_\infty\frac{r^N}{N},$$
$$(v-w)'(r)\le ||f_1||_\infty\frac{r}{N},$$

$$ \frac{\delta}{4} \le ||f_1||_\infty\frac{\bar r^2}{2N}, $$
implying $$\bar r\geq \sqrt{\frac{N\delta}{2||f_1||_\infty}} \geq 2R>r_1.$$

Let now $r_b$, $r_v$ and $r_w$ be such that $u(r_b)=b$, $u(r_v)=v(r_v)$ and $u(r_w)=w( r_w)$.

Suppose now that neither $v$, nor $w$ intersect $u$ again before $r_b$. So $u<v$ in $(r_v,r_b]$ and   $u<w$ in $( r_w,r_b]$. We will get a contradiction, and so prove our result.

  By multiplying the equation satisfied by $u$ by $(v-b)$ and the equation satisfied by $v$ by $(u-b)$, then subtracting and integrating the difference obtained over $(r_v,r_b)$ it can be easily verified that
$$-r_v^{N-1}(v-u)'(r_1)(u(r_v)-b)-(v-u)(r_b)r_b^{N-1}u'(r_b)=$$
$$\int_{r_v}^{r_b}t^{N-1}\big((f(u)-f(v))(u-b)+f(u)(v-u)\big)dt.$$
The integrand in this last integral is negative because of the super-linear condition $(H_3)$ and the fact that $u<v$ in the interval:
$$(f(u)-f(v))(u-b)+f(u)(v-u)=(u-b)(v-b)\left(\frac{f(u)}{u-b}-\frac{f(v)}{v-b}\right)<0.$$

 Hence, if  $\bar{\bar r} =min\{ \bar r,r_a\}$, where $u(r_a)=a_2$ as in Lemma \ref{u}, then
 $$-r_1^{N-1}(v-u)'(r_v)(u(r_v)-b)\leq \int_{r_v}^{r_b}t^{N-1}\big((f(u)-f(v))(u-b)+f(u)(v-u)\big)dt $$.
$$\leq \int_{r_v}^{\bar{\bar r}}t^{N-1}(u-b)(v-b)\left(\frac{f(u)}{u-b}-\frac{f(v)}{v-b}\right)dt \leq \int_{ r_v}^{\bar {\bar r}}t^{N-1}(u-b)(v-b)\left(\frac{f(w)}{w-b}-\frac{f(v)}{v-b}\right)dt $$ $$  \leq (a_2-b)^2  \int_{ r_v}^{\bar {\bar r}}t^{N-1}\left(\frac{f(w)}{w-b}-\frac{f(v)}{v-b}\right)dt <0.  $$

By the same argument for $w$ we have

$$- r_w^{N-1}(w-u)'( r_w)(u( r_w)-b)\leq (a_2-b)^2  \int_{\bar {\bar r}}^{r_a}t^{N-1}\left(\frac{f(v)}{v-b}-\frac{f(w)}{w-b}\right)dt<0.  $$

 So $$ -r_v^{N-1}(v-u)'(r_v)(u(r_v)-b)- r_w^{N-1}(w-u)'( r_w)(u( r_w)-b) \leq $$
 $$- (a_2-b)^2 \int_{ r_v }^{r_a}t^{N-1}\left|\frac{f(v)}{v-b}-\frac{f(w)}{w-b}\right|dt$$

 $$\leq   - (a_2-b)^2 \int_{  (C_2)/2 }^{C_2}t^{N-1}\left|\frac{f(v)}{v-b}-\frac{f(w)}{w-b}\right|dt = -K $$

with $C_2$ the constant in lemma \ref{u} , so K is a positive   constant  independent of $u$.

On the other hand, as
$$ -r_v^{N-1}(v-u)'(r_v)(u(r_v)-b)- r_w^{N-1}(w-u)'(r_w)(u( r_w)-b) $$
$$\ge -r_v^{N-1}|u'(r_v)|(u(r_v)-b)-r_w^{N-1}|u'( r_w)|(u( r_w)-b)$$
and both $u(r_v)-b$ and $u( r_w)-b$ are less than $\alpha_*-b$, by Lemma \ref{u}, it holds that
$$ -r_v^{N-1}(v-u)'(r_v)(u(r_v)-b)- r_w^{N-1}(w-u)'(r_w)(u( r_w)-b) \geq -2 (\alpha_*-b)CR^{N-2}(1+R^2).$$

So if we chose $R$ sufficiently small, so that the left side is greater than $-K/2$, we get a contradiction, and we have proved the lemma.

\end{proof}

\begin{lemma}(Step 2)\label{s_0}
Let $v$ be the solution found in Step 1 that intersects $u$ at least two times. Let $r_1$ and $r_2$ denote the first two intersection points and set $u(r_1)=s_1=v(r_1)$, $u(r_2)=s_2=v(r_2)$, with $s_1>s_2$, see Figure \ref{fig-4}.
If $s_2>\beta$  then there is a point $s_0 \in (s_2,s_1)$, such that

i) $\dis \frac{u'(r(s_0))}{r(s_0)}=\frac{v'(\rho(s_0))}{\rho(s_0)}$  where       $r(s)$ and $\rho(s)$ denote the inverse of $u(r)$ and $v(r)$ respectively.

ii) $P(s_0,v)<P(s_0,u)$.
\end{lemma}

\begin{figure}[h]
	\includegraphics[scale=1]{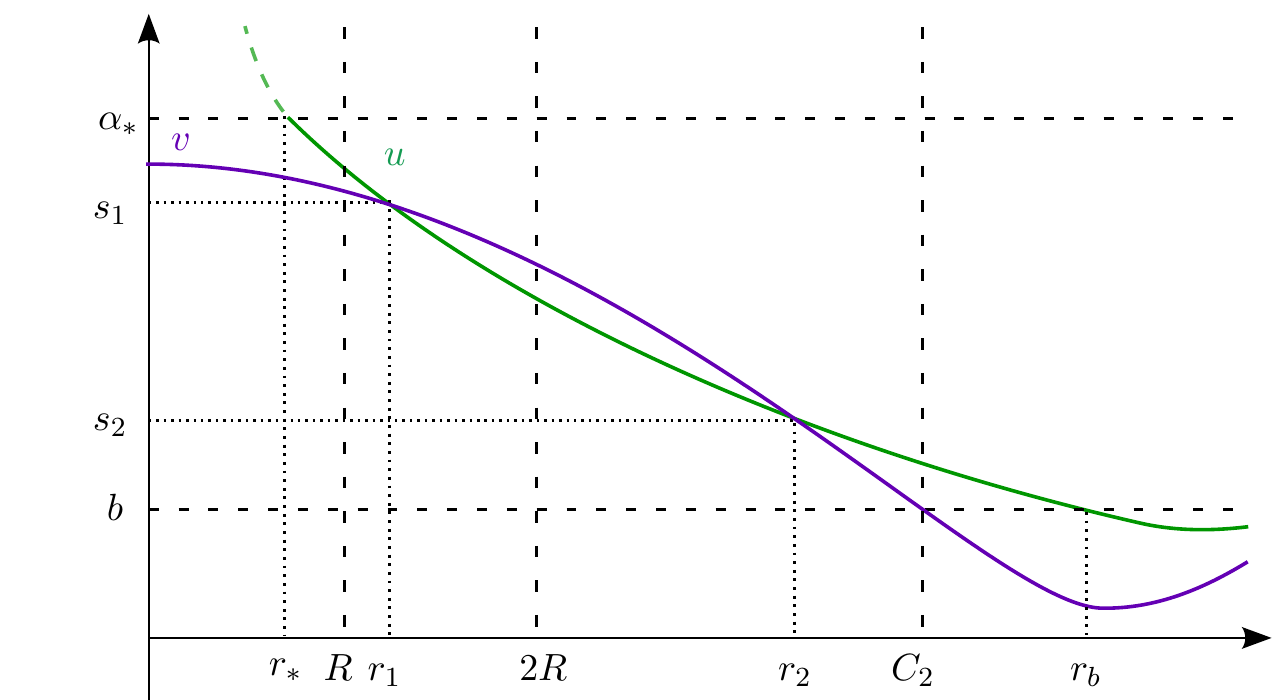}
	\caption{}\label{fig-4}
\end{figure}

\begin{proof}	
As in $s_1$ we have $|v'|<|u'|$ and in $s_2$ we have $|v'|>|u'|$, it holds that\\
at $s_1$ $\dis \frac{|u'|}{r_1}>	\frac{|v'|}{r_1}=\frac{|v'|}{r_2}$ and
at $s_2$ $\dis \frac{|u'|}{r_1}<	\frac{|v'|}{r_1}=\frac{|v'|}{r_2}$, proving $i)$. Also, from the fact that $\rho(s_0)>r(s_0)$, we obtain that $|u'(r(s_0))|<|v'(\rho(s_0))|$ and hence
\begin{eqnarray*}
	 P(s_0,u)&=&2r^N\left(N\frac{F}{f}(s_0)\frac{|u'(r(s_0))|}{r(s_0)}-\frac{|u'(r(s_0))|^2}{2}-F(s_0)\right)\\
	&>&2r^N\left(N\frac{F}{f}(s_0)\frac{|v'(\rho(s_0))|}{\rho(s_0)}-\frac{|v'(\rho(s_0))|^2}{2}-F(s_0)\right).
\end{eqnarray*}

We observe now that by $(H_2)$, $P(\cdot, v)$ is increasing, and since $P(v(0),v)=0$, it follows that $P(s,v)<0$, hence, as $\rho(s_0)>r(s_0)$, we obtain
\begin{eqnarray*}
P(s_0,u)>	2r^N\left(N\frac{F}{f}(s_0)\frac{|v'(\rho(s_0))|}{\rho(s_0)}-\frac{|v'(\rho(s_0))|^2}{2}-F(s_0)\right)\\
	>2\rho^N\left(N\frac{F}{f}(s_0)\frac{|v'(\rho(s_0))|}{\rho(s_0)}-\frac{|v'(\rho(s_0))|^2}{2}-F(s_0)\right)
	=P(s_0,v)
\end{eqnarray*}
implying $ii)$.
\end{proof}

\begin{lemma}(Step 3)\label{n-1} Let $s_0$ be the point in Step 2, then
\begin{equation}\label{m3} (r(s))^{N -1}|u'(r(s))|\leq (\rho(s))^{N -1}|v'(\rho(s))|\quad\mbox{in}\quad [s_2,s_0].
	\end{equation}
\end{lemma}
\begin{proof}
As 	$\rho(s)>r(s)$ in $(s_2,s_1)$, \eqref{m3} clearly holds at $s=s_0$ and
\begin{eqnarray*}( \rho^{N-1}v')^2\frac{d}{ds}\Bigl(\frac{r^{N-1}u'}{\bar r^{N-1}v'}\Bigr)&=&-f(s)\Bigl(\frac{\rho^{N-1}v'}{r^{N-1}u'}\Bigr)\Bigl(r^{2(N-1)}-\rho^{2(N-1)}\Bigl(\frac{r^{N-1}u'}{\rho^{N-1}v'}\Bigr)^2\Bigr)\\
&>&-f(s)\Bigl(\frac{\rho^{N-1}v'}{r^{N-1}u'}\Bigr)\Bigl(1-\Bigl(\frac{r^{N-1}u'}{\rho^{N-1}v'}\Bigr)^2\Bigr)\rho^{2(N-1)}.
\end{eqnarray*}
Hence, if at some point $\bar s\in(s_2,s_0)$ it holds that
$(r(\bar s))^{N -1}|u'(r(\bar s))|= (\rho(\bar s))^{N -1}|v'(\rho(\bar s))|$
and $(r(\bar s))^{N -1}|u'(r(\bar s))|< (\rho(\bar s))^{N -1}|v'(\rho(\bar s))|$ in $(\bar s,s_0)$, then
$$\frac{d}{ds}\Bigl(\frac{r^{N-1}u'}{\rho^{N-1}v'}\Bigr)(\bar s)>0,$$
a contradiction, and thus \eqref{m3} follows.
	\end{proof}
\begin{lemma}(Step 4)\label{proof} End of the proof  of proposition \ref{EnP-f1}.
\end{lemma}
\begin{proof}
	Finally we observe that by $(H_2)$ and the previous lemma we have
$$(P(\cdot,v)-P(\cdot,u))_s=\left(N-2-2N\Bigl(\frac{F}{f}\Bigr)'(s)\right)\Bigl(\bar r^{N-1}v'(\rho(s))-r^{N-1}u'(r(s))\Bigr)>0$$
in $[s_2,s_0]$, therefore, from Lemma \ref{s_0} (ii) we conclude that, if $s_2>\beta $ then
$P(s_2,v)<P(s_2,u)$. We are now in the situation of Lemma \ref{final1} with $u_1=u$, $u_2=v$ and $s_I=s_2$ so we conclude that $u$ does not intersect $v$ again and $u(r)>v(r)>0$ for all $r>r (s_2)$.
As $v\in {\mathcal P}$ , we have $u(r)>0$ for all $r>r_*$.
\end{proof}

\section{Proof of main Theorems}\label{sec f2}

In this section we analyze the solutions to \eqref{ivp} for $f$ defined by
 \begin{eqnarray}\label{ffinal}
f(s)=\begin{cases}
f_1(s) &  s\leq \alpha_1=\alpha_*+\epsilon_1\\
L_1(s) & \alpha_1 \leq s\leq \alpha_1+\epsilon_1\\
A_2^2f_2(s) &  s\geq \alpha_1+\epsilon_1
\end{cases}
\end{eqnarray}
and determine the existence of $\epsilon_1$ and $A_2$ so that problem \eqref{pde}  has at least two solutions. The following result is the crucial ingredient of the proof of Theorem \ref{second solution}.

\begin{proposition}\label{enP-Agrande} There is an $\alpha_0>\alpha_*$ such that for any  $\alpha\in(\alpha_*,\alpha_0)$, there is an $\bar\epsilon>0$ and $\bar A>0$ such that for any $\epsilon_1<\bar\epsilon$ and $A_2>\bar A$ the solution to  problem \eqref{ivp} with $f$ defined in \eqref{ffinal}
with
$u(0)=\alpha$ and $u'(0)=0$, is positive for all $r$.
\end{proposition}

We will prove this using Proposition \ref{EnP-f1}. By choosing small enough $\alpha$ and $\epsilon_1$, and big enough $A_2$, we will be able to prove that the solution with initial value $\alpha$ will reach $\alpha_*$ at $r_*<R$ with $r_*|u'(r_*)|$ bounded.             In order to prove this result we need the following lemma.

\begin{lemma}\label{epsilon}
Let $v$ be the solution to the initial value problem
\begin{eqnarray}\label{ge}
\begin{gathered}
v''+\frac{N-1}{r}v'+g(v)=0,\quad r>r_\delta,\\
v(r_\delta)=\bar\alpha +\delta,\quad v'(r_\delta)=v'_\delta<0
\end{gathered}
\end{eqnarray}
where $g$ is a positive continuous function defined in $[\bar\alpha,\bar\alpha+\delta]$ and $r_\delta>0$. Let $\bar r$ be defined by  $v(\bar r)=\bar\alpha$, $||g||_+ = \max \{g(s) : s \in [\bar\alpha,\bar\alpha +\delta]\}$ and $a>\delta$.

  If  $2(N-2)a<r_{\delta}|v'(r_\delta)|$, then    $\bar r<2 r_\delta$ and $$(N-2)a\le \bar r|v'(\bar r)|\le r_\delta|v'(r_\delta)|+\frac{2r_\delta^2||g||_+}{(N-2)a}\ \delta .$$
\end{lemma}

\begin{proof}
Since $v$ is a solution of \eqref{ge}, we have that for $r\in[r_\delta,\bar r ]$
\begin{equation}\label{m0}
	r^{N-1}v'(r)= r_\delta^{N-1}v'(r_\delta) -\int_{r_\delta}^r t^{N-1}f_\delta(v(t))dt < r_\delta^{N-1}v'(r_\delta)
\end{equation}
dividing by $r^{N-1}$ and integrating over $[r_\delta,\bar r ]$ we get
\begin{equation*}
	\delta= v(r_\delta)-v(\bar r )> \frac{r_\delta^{N-1}|v'(r_\delta)|}{N-2}\left( \frac{1}{r_\delta^{N-2}} -  \frac{1}{\bar r ^{N-2}}\right)
\end{equation*}
implying
\begin{equation}\label{m1}
	\delta> \frac{r_\delta |v'(r_\delta)|}{N-2}\left( 1 -  \frac{r_\delta^{N-2}}{\bar r^{N-2}}\right) >2a\left( 1 -  \frac{r_\delta^{N-2}}{\bar r^{N-2}}\right).
\end{equation}
Since $a>\delta$ we conclude
$\bar r < 2^{1/{N-2}} r_\delta< 2 r_\delta$. Moreover, from \eqref{m0},
$$\bar r|v'(\bar r)|\ge \Bigl(\frac{r_\delta}{\bar r}\Bigr)^{N-2}r_\delta |v'(r_\delta)|\ge\frac{1}{2}r_\delta |v'(r_\delta)|\ge (N-2)a.$$
Finally, from \eqref{m1} we obtain
$$\frac{\delta}{2a}\ge \left( 1 -  \frac{r_\delta^{N-2}}{\bar r^{N-2}}\right)$$
and thus
$$\bar r^{N-2}-r_\delta^{N-2}\le \frac{\delta}{a}r_\delta^{N-2},$$
and by the mean value theorem,
$$(N-2)r_\delta^{N-3}(\bar r-r_\delta)\le \frac{\delta}{a}r_\delta^{N-2}$$
implying
$$\bar r-r_\delta\le \frac{1}{(N-2)}\frac{\delta}{a}r_\delta.$$
Using again \eqref{m0} and since $r_\delta<\bar r<2r_\delta$, we obtain
$$\bar r|v'(\bar r)|\le r_\delta|v'(r_\delta)|+\frac{2r_\delta^2||g||_+ }{(N-2)a}\ \delta.
$$

\end{proof}

\begin{proof}[Proof of Proposition \ref{enP-Agrande}]

We consider first the solutions $w$ of problem \eqref{ivp} with $f=f_2$ extended continuously below $\alpha_*$, and initial conditions $w(0,\alpha)=\alpha\geq \alpha_*$ and $w'(0,\alpha)=0$. Let $r_*(\alpha)$ be the radius were $w(r_*(\alpha),\alpha)=\alpha_*$, then the expression $r_*(\alpha)|w'(r_*(\alpha),\alpha)|$ is continuous with $r_*(\alpha_*)|w'(r_*(\alpha_*),\alpha_*)|=0$, thus there is an $\alpha_0>\alpha_*$ such that for $\alpha\in(\alpha_*,\alpha_0)$ we have \begin{equation*} r_*(\alpha)|w'(r_*(\alpha),\alpha)| < \bar b/3
\end{equation*}
where we choose $\bar b$ of Proposition \ref{EnP-f1} as $\bar b= \frac{(\alpha_*-b) (N-2) }{2}$.

Fixing  $\alpha\in(\alpha_*,\alpha_0)$, denote by $r_*^w=r_*(\alpha)$ and $r_\epsilon^w$ the radius were $w(r_\epsilon^w,\alpha)=\alpha_*+2\epsilon$.  Again by continuity, there is an $\bar \epsilon<\alpha/2$ such that for $\epsilon<\bar \epsilon$
\begin{equation}\label{wprima} r_\epsilon(\alpha)|w'(r_\epsilon(\alpha),\alpha)| < \bar b/2.
\end{equation}

By choosing a smaller $\bar \epsilon$ if necessary, we can assume that there is an $a>0$  such that for all $\epsilon< \bar \epsilon$,
$$2\epsilon<a<\frac{r_{\epsilon}^w |w'(r_{\epsilon}^w,\alpha)|}{2(N-2)}.$$
 This can be achieved since the right hand side does not tend to $0$ with $\epsilon$.

We will also require $\bar \epsilon$ to be  small enough so that
$$\frac{2(r_\epsilon^w)^2 f_\infty \epsilon}{(N-2)a}<\frac{2  (r_*^w)^2f_\infty\bar \epsilon}{(N-2)a}< \bar b/2$$
where $f_\infty$ denotes the maximum value of $f_1$ or $f_2$ on the interval $(\alpha_*,\alpha)$.

 Now let $u$ be the solution of problem \eqref{ivp} with $f$ as in the statement, with the $\alpha$ and $\epsilon_1<\bar \epsilon$ chosen as above, and initial conditions $u(0)=\alpha$ and $u'(0)=0$.

Note that for any such solution, the function $v(r)= u(r/A)$ satisfies:
$$v''(r)+\frac{(N-1)}{r}v'(r)= \frac{1}{A^2}\left(u''(r/A)+\frac{(N-1)}{r/A}v'(r/A)\right)= -\frac{1}{A^2}f(v(r))$$
so for $r<r_\epsilon^v$, (where $v(r_\epsilon^v)=\alpha_1+\epsilon_1$), $v$ coincides with $w$. Moreover, $r_\epsilon^v= A r_\epsilon^u$ and $v'(r_\epsilon^v)= u'(r_\epsilon^v/A)/A =  u'(r_\epsilon^u)/A$ hence $r_\epsilon^u|u'(r_\epsilon^u)|$ is independent of $A$.

Choose $\bar A>1$ such that $r_*^w/\bar A<R/2$, where $R$ is the constant from Proposition \ref{EnP-f1} with $\bar a=(N-2)a$. Then for any $A_2>\bar A$ the solution $u(r)$ will be $u(r)= w(A_2r)$ for $r<r_\epsilon^u$, and at this point the choice of $\epsilon_1$ guarantees that  $$2(N-2)a<r_\epsilon^u|u'(r_\epsilon^u)|=r_\epsilon^w|w'(r_\epsilon^w)|<\bar b/2$$ thus it satisfies the condition of Lemma \ref{epsilon}. Therefore $r_* < 2r_\epsilon^u = 2r_\epsilon^w/A_2< 2r_*^w/A_2<R$ and

 $$\bar a=(N-2)a\le r_*|v'(r_*)|\le r_\epsilon^u|v'(r_\epsilon^u)|+\frac{2(r_\epsilon^u)^2||g||_+ }{(N-2)a}\ \epsilon \leq \bar b/2 +\frac{2(r_\epsilon^w)^2f_\infty }{(N-2)a}\ \epsilon < \bar b $$
were we are using that the maximum value of $f$ in $(\alpha_*,\alpha_*+2\epsilon)$ is either the maximum of $f_1$ or $A_2^2f_2$ thus $(r_\epsilon^u)^2f\leq \max\{(r_\epsilon^v)^2f_1/A_2^2, (r_\epsilon^v)^2f_2\}\leq (r_\epsilon^v)^2 f_\infty\leq (r_\epsilon^w)^2 f_\infty$.

Thus, $u$ satisfies the conditions of Proposition \ref{EnP-f1}, and it is positive for all $r$.

\end{proof}

\begin{figure}[h]
	\includegraphics[scale=1]{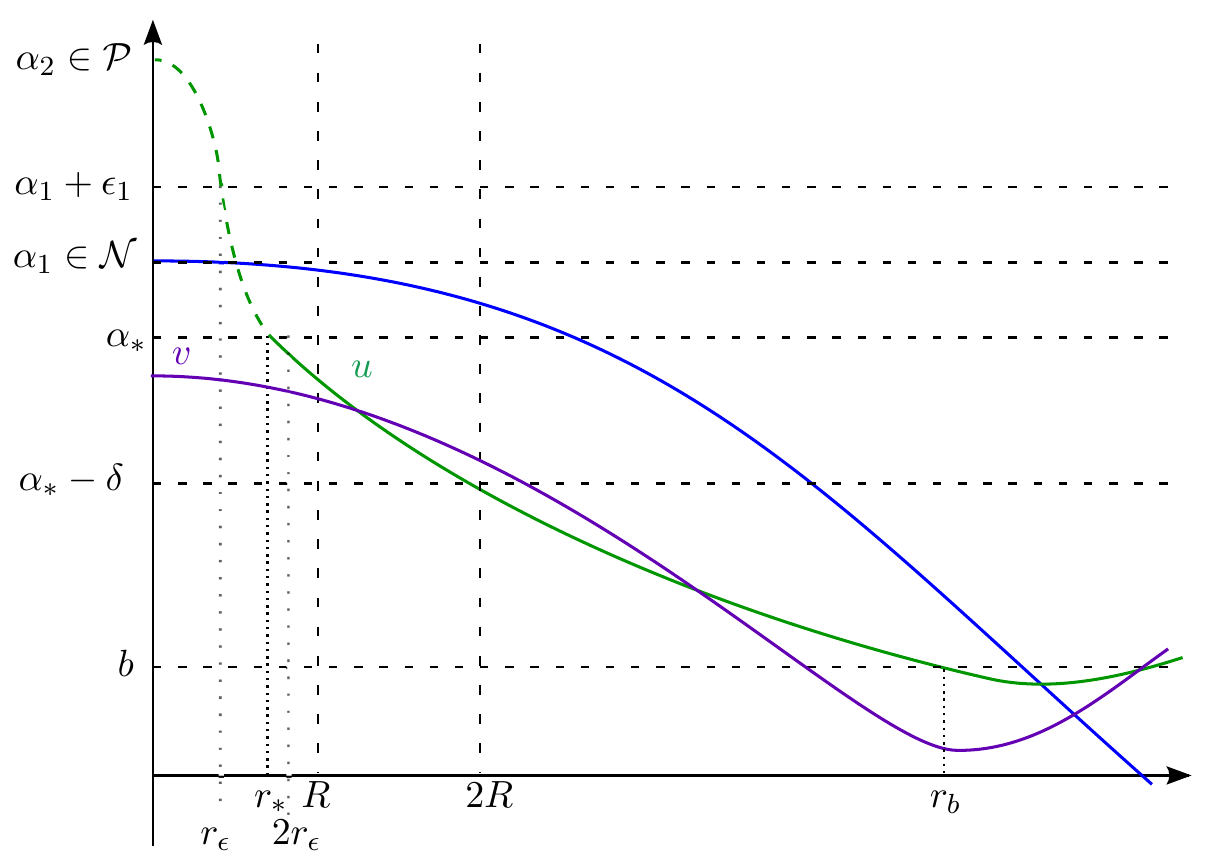}
	\caption{}\label{fig-5}
\end{figure}

\subsection*{Proof of Theorem \ref{second solution}}

Let $\alpha_0$ as in Proposition \ref{enP-Agrande} and $\alpha\in(\alpha_*,\alpha_0)$, let $\bar A$ and $\bar \epsilon$ be the constants given in Proposition \ref{enP-Agrande}. Given  $f$ as in \eqref{f}, with $A_2>\bar A$ and $0<\epsilon_1<\bar \epsilon$ we will consider  solution $u$ of problem \eqref{pde} with $u(0)=\alpha$, $u'(0)=0$.

It is known that for subcritical functions $f_1$ satisfying the conditions $(H_1)$-$(H_4)$, the solutions are positive for $\alpha<\alpha_*$ and change sign for $\alpha>\alpha_*$,  see Proposition \ref{open-sets}. Therefore, a solution to our problem will be positive for $\alpha<\alpha_*$ and changes sign for $\alpha_*<\alpha\leq \alpha_1$.

Moreover, from Proposition \ref{enP-Agrande} we have that, for some $\alpha_2>\alpha_1+\epsilon_1$, the solution is positive, that is $\alpha_2\in\mathcal P$.
Therefore, since $\mathcal N$ and $\mathcal P$ are open disjoint sets, between $\alpha_1 \in \mathcal N$ and $\alpha_2\in \mathcal P$ there must be an $\alpha_2^*$ that is in neither, this corresponds to a second ground state solution.

\subsection*{Proof of Theorem \ref{third solution}}
It follows from the results in \cite{cku} or section 4 in \cite{gst} that if condition $(H_6)$ is satisfied, then $\alpha\in\mathcal N$ for $\alpha$ large enough, hence there must be another value $\alpha_3^*>\alpha_2$ where $\alpha_2\in\mathcal P$ is the value given in the proof of Theorem \ref{second solution} such that $u(\cdot,\alpha_3^*)$ is a ground state solution.

\subsection*{ Proof of Theorem  \ref{solution f_3}}\label{sec f3}

We know there is a constant $K=K(f_1)$ such that if for an initial condition $\alpha$, the corresponding solution is such that for $ u(r)=\alpha_*$, $r>K$, then
$\alpha$ is in ${\mathcal N}$ See \cite[Lemma 3.1]{gst}. We will prove that such  $\alpha$ exists, choosing $\epsilon_2$ and $A_3$ .

Choose $\epsilon_2< \frac{\gamma-\alpha_*}{4}$, then by Proposition \ref{enP-Agrande} there is an $\alpha_2<\alpha_*+\epsilon_2$ and an $\epsilon_1$ such that $\alpha_1=\alpha_*+\epsilon_1<\alpha_2$ and  $\alpha_2 \in \mathcal P$. Note that $\alpha_2+3\epsilon_2<\gamma $ and define $||f_3|| =  \max \{f_3(s) : s \in [\alpha_2,\alpha_2+3\epsilon_2 ]\}$.

Now take $\alpha$ such that  $ \alpha_2+2\epsilon_2 <\alpha < \alpha_2+3\epsilon_2$ and let $u$ the solution with this initial condition and $r_\epsilon$ the radius where $u(r_\epsilon)=\alpha_2+\epsilon_2$. Then for $r\leq r_\epsilon$
$$|r^{N-1}u'(r)|=\left|\int_{0}^rt^{N-1}A_3^2f_3\ dt\right|\le A_3^2||f_3|| \frac{r^N}{N} $$
 and $$|u'(r)|\le A_3^2||f_3|| \frac{r}{N} $$
hence integrating in $(0,r_\epsilon)$  we find that
$$ \epsilon_2< \alpha -(\alpha_2+\epsilon_2)\le A_3^2||f_3|| \frac{r_\epsilon}{N}. $$
Choose $A_3$ sufficiently small such that $r_\epsilon>K$. As $r$ is increasing, $r>K$ for  $ u(r)=\alpha_*$   and so $\alpha$ is in ${\mathcal N}$.
As  ${\mathcal  P}$ and  ${\mathcal N}$ are open disjoint sets, there is a ground state in the interval $( \alpha_2,\alpha )$.

\section{Proof of Theorem \ref{multiple} }\label{fk}

Following the ideas in the proofs of Theorems \ref{second solution} and \ref{solution f_3}  we will prove  for each $i$ the existence of the constants $\epsilon_{i-1},\ A_i$ and $\alpha_i$ such that  $\alpha_i \in \mathcal  P$ if $i$ is even and $\alpha_i \in \mathcal  N$ if $i$ is odd. As   $\mathcal  P$  and $ \mathcal  N$
are open and disjoint, there is an $\alpha $ in between two consecutive $\alpha_i $ that gives rise to a ground state solution.

Let $\alpha_0 $ be  such that  if  $ I=[ \alpha_*, \alpha_0 ]$, $||f||_+ = \max \{|f(s)| : s \in I \}$ and  $||f||_- = \min \{|f(s)| : s \in I \}$ then

$$ \frac{||f_i||_+}{||f_i||_-} \leq 3/2 $$
for all $i$.

Take $$\epsilon_0=\frac{ \min \{(\alpha_0-\alpha_*), ((\alpha_*-b)(N-2)/6)\}}{(3N)^k}$$

For our method to work we will actually prove recursively the following statement:
\medskip

\noindent {\bf Claim:} There exist $\epsilon_{i-1}$, $\alpha_i$ and $A_i$ such that $\alpha_{i-1}+\epsilon_{i-1}\leq \alpha_i\leq \alpha_*+\epsilon_0 (3N)^i$ and  $\alpha_i \in \mathcal  P$ if $i$ is even and $\alpha_i \in \mathcal  N$ if $i$ is odd.
\medskip

 For $i=2$  see Proposition \ref{enP-Agrande} and choose $\alpha_2\leq \alpha_*+\epsilon_0(3N)^2$.
 We will also use the following result.

\begin{lemma}\label{delta}
Let $u$ be the solution to the initial value problem
\begin{eqnarray}\label{g}
\begin{gathered}
u''+\frac{N-1}{r}u'+g(u)=0,\quad r>0,\\
u(0)=\bar\alpha +\delta,\quad u'(0)=0
\end{gathered}
\end{eqnarray}
where $g$ is a positive continuous function defined in $[\bar\alpha,\bar\alpha+\delta]$ where $\delta>0$. If $\bar r$ is defined by  $u(\bar r)=\bar\alpha$,
 $||g||_+ = \max \{g(s) : s \in [\bar\alpha,\bar\alpha +\delta]\}$ and $||g||_- = \min \{g(s) : s \in [\bar\alpha,\bar\alpha +\delta] \},$
then
\begin{enumerate}
\item[i)]
 $$\sqrt{\frac{2N\delta }{||g||_+}}\leq \bar r\leq \sqrt{\frac{2N\delta }{||g||_-}}$$
\item[ii)]
$$ 2\delta \frac{||g||_- }{||g||_ +}  \leq \bar r|u'(\bar r)|\leq 2\delta\frac{||g||_+ }{||g||_ -} $$
\end{enumerate}
\end{lemma}
\begin{proof}
 Writing  \eqref{g}  as
$$-(r^{N-1}u'(r))'=r^{N-1}g(u),$$
and integrating over $(0,r)$, $r\le\bar r$ we obtain
$$r^{N-1}|u'(r)|=\int_{0}^rt^{N-1}g(u)dt$$
and thus
$$\frac{r}{N}||g||_-  \leq |u'(r)|\leq \frac{r}{N}||g||_+.  $$
Integrating this last inequality over $(0,\bar r)$ we get
$$\frac{(\bar  r)^2}{2N}||g||_-  \leq \delta \leq \frac{(\bar r)^2}{2N}||g||_+ . $$
Hence
$$\sqrt{\frac{2N\delta }{||g||_+}}\leq \bar r\leq \sqrt{\frac{2N\delta }{||g||_-}}$$
and
$$ 2\delta \frac{||g||_- }{||g||_ +} \leq \frac{(\bar r)^2}{N}||g||_-  \leq \bar r|u'(\bar r)|\leq \frac{(\bar r)^2}{N}||g||_+ \leq 2\delta\frac{||g||_+ }{||g||_ -}. $$

 \end{proof}

\begin{proposition}\label{odd}

Let $i$ odd. Given $\epsilon_{j-1},\ A_j$ and $\alpha_j$ satisfying the above Claim for $j=2,\ldots,i-1$,  Then the Claim is also satisfied for $i$.
 \end{proposition}

\begin{proof}
As in the proof of Theorem \ref{solution f_3} we will look for an $\alpha$ so that the solution with $u(0)=\alpha$  reaches  $\alpha_*$ at $ u(r)=\alpha_*$, with $r>K(f_1)$. We will prove that such  $\alpha$ exists, choosing $\alpha_i$ and $A_i$ .

 Let $\epsilon_{i-1}= \epsilon_0 $,
let $\alpha_i$ such that  $\alpha_{i-1}+2\epsilon_0<\alpha_i <\alpha_{i-1}+3\epsilon_0 $ and let $u$ be the solution to \eqref{ivp} with $r_0=0$ with this initial condition, that is $u(0)=\alpha_i,\ u'(0)=0$.  By Lemma \ref{delta}

 $$ r_{i} \geq \sqrt{\frac{ 2N(\alpha_i - (\alpha_{i-1}+\epsilon_0))}{ A_i^2||f_i||_+}}>\sqrt{\frac{ 2N\epsilon_0}{ A_i^2||f_i||_+}}$$
where $u(r_i)=\alpha_{i-1}+\epsilon_0$.
Choose $A_i$ sufficiently small such that $r_i>K$. As $r$ is increasing, $r>K$ for  $ u(r)=\alpha_*$   and so $\alpha$ is in ${\mathcal N}$.

On the other hand

 $$ \alpha_i < \alpha_{i-1}+3\epsilon_0 < \alpha_*+\epsilon_0 (3N)^{i-1}+3\epsilon_0 <\alpha_*+\epsilon_0 (3N)^{i},$$
proving the proposition.
\end{proof}

\begin{proposition}\label{even}

Let $i$ even. Given $\epsilon_{j-1},\ A_j$ and $\alpha_j$ satisfying the above Claim for $j=2,\ldots,i-1$,  Then the Claim is also satisfied for $i$.
 \end{proposition}

\begin{proof} Let $\alpha_{i}= \alpha_*+\epsilon_0 (3N)^{i}$, we  will use Proposition \ref{EnP-f1} to prove that $\alpha_i\in \mathcal  P$, for this we need bounds on $r_*$ and $r_*|u'(r_*)|$, where we recall $r_*$ is the point such that $u(r_*)=\alpha_*$.

Let $\epsilon_{i-1}<\epsilon_0$ to be determined later, and let $u$ be the solution with initial condition $u'(0)=0$, $u(0)=\alpha_i$. Let $r_i$ such that  $u(r_i)=\alpha_{i-1}+\epsilon_{i-1}$, see Figure \ref{casoi}.
\begin{figure}[h]
  \includegraphics[scale=1]{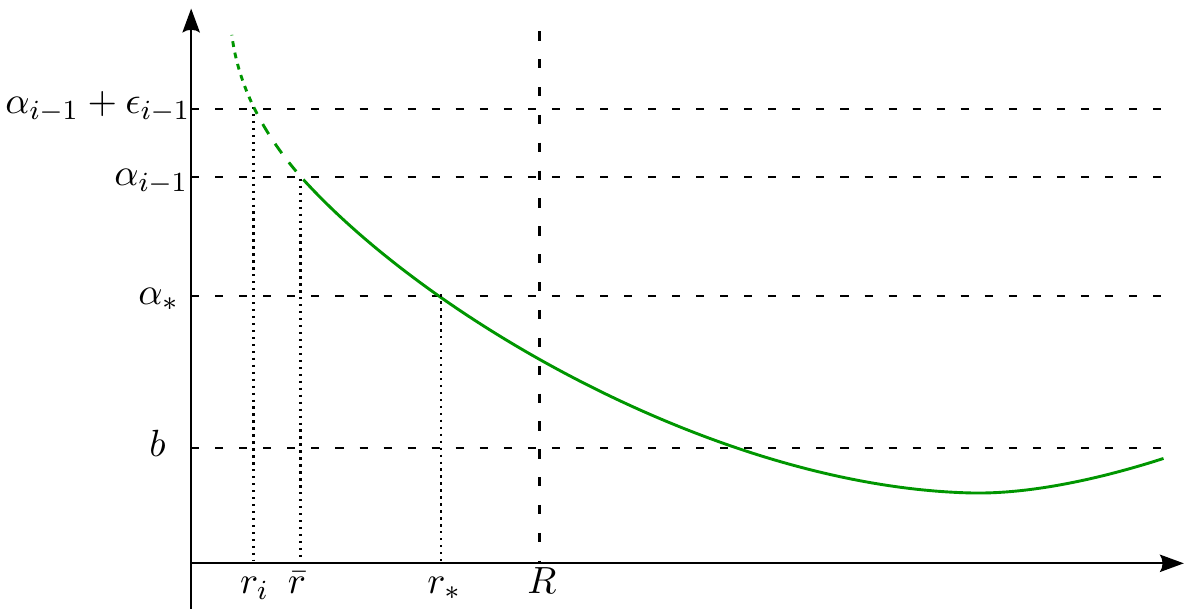}
  \caption{ }\label{casoi}
\end{figure}

\noindent {\bf Step 1. } [Lower bound] We will prove that
$$ r_i|u'( r_i)|> 4 (N-2)(\alpha_{i-1}+\epsilon_0-\alpha_*)\epsilon_0 $$
and
$$ r|u'( r)|>2 (N-2) ((3N)^{i-1}+1)\epsilon_0> 2(N-2)(\alpha_{i-1}+\epsilon_0-\alpha_*)>0\quad\mbox{ for all $ r\in [r_i, r_*]$.} $$
Indeed, we have
$$(\alpha_i-(\alpha_{i-1}+\epsilon_{i-1}) )>((3N-1) (3N)^{i-1}-1)\epsilon_0 >
3 (N-2)((3N)^{i-1}+1)\epsilon_0,$$
then as $\frac{||A_i^2f_i||_- }{||A_i^2f_i||_ +}=\frac{||f_i||_- }{||f_i||_ +}\geq 2/3$, using  Lemma \ref{delta},
$$ r_i|u'( r_i)|>4 (N-2) ((3N)^{i-1}+1)\epsilon_0> 4 (N-2)(\alpha_{i-1}+\epsilon_0-\alpha_*). $$
Now Using  Lemma \ref{epsilon}, it holds that for any $ r\in (r_i, r_*]$
$$ r|u'(r)|>2 (N-2) ((3N)^{i-1}+1)\epsilon_0> 2(N-2)(\alpha_{i-1}+\epsilon_0-\alpha_*)>0. $$
\medskip

\noindent {\bf Step 2. } [Upper bounds at $\alpha_{i-1}$ ]
Let $\bar r$ such that $u(\bar r)=\alpha_{i-1} $.
We will prove that
$$\bar r|u'( \bar r)|<\frac{3(\alpha_* -b)(N-2)}{4},$$ and $$\bar r \leq 2 \frac{c_i}{A_i}$$
for some positive constant $c_i$ independent of $A_i$.

Indeed, by Lemma \ref{delta}\ i) we can see that $$r_i \leq \frac{c_i}{A_i}$$ where $c_i$ is independent of $A_i$
and
$$ r_i|u'( r_i)|< 2(\alpha_i-(\alpha_{i-1}+\epsilon_{i-1}) )\frac{||f_i||_+ }{||f_i||_ -}< 3  (\epsilon_0 (3N)^{k})\leq (\alpha_*-b)(N-2)/2,$$
which is also independent of $A_i$.

Let  $A^2_i\geq A^2_{i-1}||f_{i-1}||_{+}/||f_{i}||_{+}$ to be determined later. Then  the function $f$ for $s \in [\alpha_{i-1},\alpha_{i}] $ satisfies
$$||f||\leq \max\{ A^2_{i-1}||f_{i-1}||_{+}, A^2_{i}||f_{i}||_{+}\}= A^2_{i}||f_{i}||_{+} $$

We can use Lemma \ref{epsilon} on $[\alpha_{i-1},\alpha_{i-1}+\epsilon_{i-1}]$, since from Step 1 we know that $r_i|u'( r_i)|$ is bounded from below. We obtain that
$$\bar r \leq 2r_{i} \leq 2 \frac{c_i}{A_i}$$
and we can find a constant $C$, independent of $A_i$, such that   $$\bar r|u'( \bar r)|<(\alpha_*-b)(N-2)/2+\frac{2r^2_i  A^2_{i}||f_{i}||_{+}}{2(N-2) \left((3N)^{i-1}+1\right)\epsilon_0}\epsilon_{i-1}< (\alpha_*-b)(N-2)/2+C\epsilon_{i-1}. $$
Choosing $\epsilon_{i-1}< \min\{\epsilon_0,\frac{(\alpha_*-b)(N-2)}{4C}\}$ we get
$$\bar r|u'( \bar r)|<\frac{3(\alpha_* -b)(N-2)}{4}.$$

\noindent {\bf Step 3. }   [Upper bounds at $\alpha_{*}$ ] \\ Let now $r_*$ be such that $u(r_*)=\alpha_*$ and  $C_i = \max \{||f_1(s)||_{+}, A^2_2||f_2(s)||_{+}, \dots , A^2_{i-1}||f_{i-1}(s)||_{+} \}$. We use again Lemma \ref{epsilon},  this time on $[\alpha_{*},\alpha_{i-1}]$, to obtain
$$ r_*|u'(  r_*)|<\bar r|u'( \bar r)|+\frac{2\bar r^2 C_i}{(N-2)((3N)^{i-1}+1)\epsilon_0 }\leq \frac{3(\alpha_*-b)(N-2)}{4}+\frac{2 \bar r^2 C_i}{(N-2)((3N)^{i-1}+1)\epsilon_0 }$$
$$\leq \frac{3(\alpha_*-b)(N-2)}{4}+\frac{ 2C_i}{(N-2)((3N)^{i-1}+1)\epsilon_0 }\left(2\frac{c_i}{A_i}\right)^2$$
and $$ r_*\leq 2\bar r \leq 4 \frac{c_i}{A_i}.$$
Hence for $A_i$ big enough we have the hypothesis of Proposition \ref{EnP-f1}, and thus $\alpha_i \in {\mathcal  P}$ .

\end{proof}

\section{Examples}\label{examples}

In this section we want to present some examples that illustrate the scope of the functions considered in Theorem \ref{second solution}, and the different behavior of solutions with large initial condition depending on the behavior of $f$ near infinity. Note that conditions $(H_1)-(H_4)$ are satisfied for $f_1(s)= s^p-s$ with $1<p<\frac{N+2}{N-2}$ subcritical and $N>2$. We will use these functions $f_1$ to construct our examples.

\subsection*{Example 1}
We first point out that for Theorem \ref{second solution} the only condition on $f_2$ is $(H_5)$, that requires $f_2$ to be positive and continuous. Thus the theorem applies to any such function, for example $f_2(s)=2+\sin (s)$. \\

\medskip
We will now consider the behavior of $f$ near infinity, and note that the function $f_2(s)=s^q$ will satisfy condition $(H_6)$ if  $q<\frac{N+2}{N-2}$ (subcritical), and will not satisfy it  if  $q>\frac{N+2}{N-2}$ (supercritical). We will use these functions to illustrate the expected behavior of solutions with  large initial condition in both cases.

We consider functions of the  following form with $1<p<\frac{N+2}{N-2}$
 \begin{eqnarray}\label{f ex}
f(s)=\begin{cases}
f_1(s)=s^p-s &  s\leq \alpha_*+\epsilon_1 =\alpha_1\\
L_1(s) & \alpha_1\leq s\leq \alpha_1+\epsilon_1\\
A_2^2f_2(s)=A_2^2s^q &  s\geq \alpha_1+\epsilon_1
\end{cases}
\end{eqnarray}
where $\alpha_*$  as in $(H_4)$ and $L_1(s)$ is the line from $(\alpha_1, f_1(\alpha_1))$ to $(\alpha_1+\epsilon_1, A_2^2f_2(\alpha_1+\epsilon_1))$.

\subsection*{Example 2}

If both $p, q<\frac{N+2}{N-2}$ this function satisfies $(H_1)-(H_6)$, an thus by Theorem \ref{third solution}, for small enough $\epsilon$ and big enough $A$, has at least three ground state solutions.  This is particularly interesting when we consider $p=q$ and compare it to the problem with $f= s^p-s$ that has a unique ground state solution, as pointed out in Proposition \ref{open-sets}.\\

In computer experiments we can see that the third solution is not necessarily too big. For $N=4$, $p=q=2$, $A_2=10$ and $\epsilon_1=1/10$, we can see there should be three solutions with $\alpha<30$ (see  Figure \ref{nuevo}).

\begin{figure}[h]
  \includegraphics[scale=1]{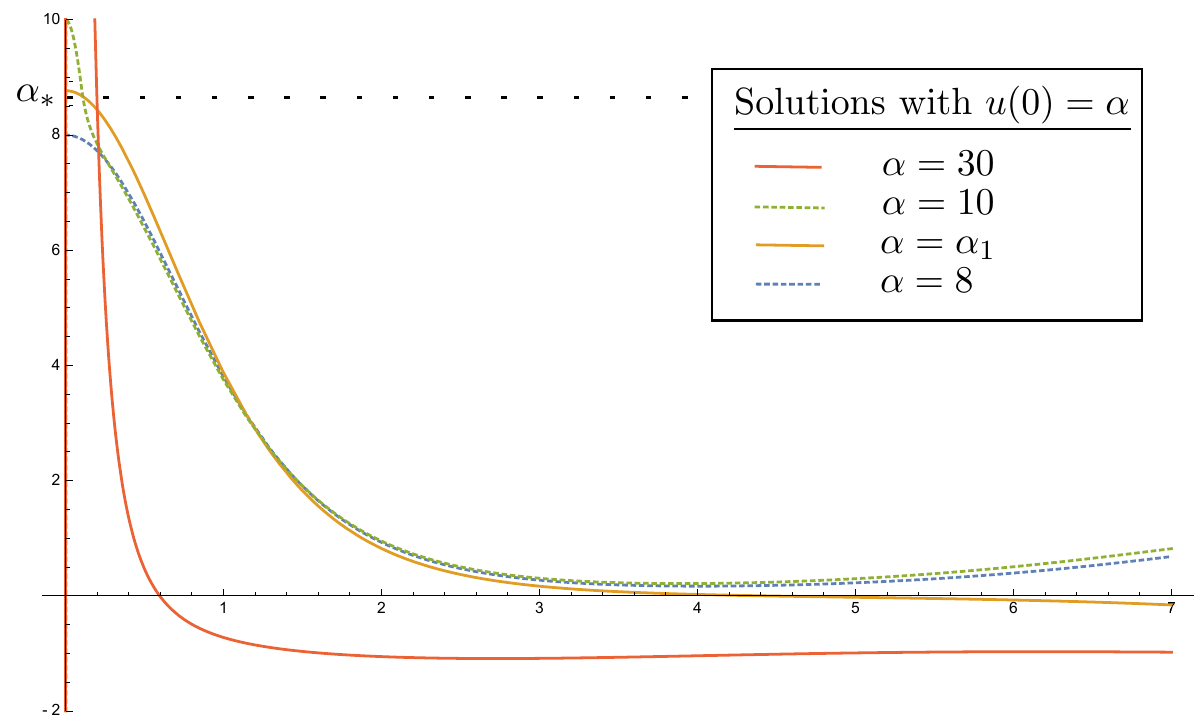}
  \caption{Solutions for $N=4$, $p=q=2$, $A_2=10$ and $\epsilon_1=1/10$.}\label{nuevo}
\end{figure}

\medskip

In Theorem \ref{third solution} we give a sub-critical  condition on $f_2$, to obtain that if $\epsilon$  is sufficiently small, $A$ and $\alpha $ big enough then $\alpha \in {\mathcal N}$, and from this fact we obtain a third
ground state.  The following example shows that some condition on $f_2$ is necessary to obtain that  $\alpha \in {\mathcal N}$ for large $\alpha$.

\subsection*{Example 3}
Let $N=4$, $p=2<\frac{N+2}{N-2}=3$ and  $q>\frac{N+2}{N-2}$ and $f$ as in $\eqref{f ex}$. This function satisfies $(H_1)-(H_5)$ but not $(H_6)$, and we will see that if $\epsilon$ is sufficiently small, $A$ and $\alpha$ big enough, then $\alpha \in {\mathcal P}$.

\begin{proof}

We will first see that $\alpha_*>3$. To this end  we  recall a Pohozaev identity which was obtained
by Ni and Serrin \cite{ni-serrin}.
$$E(r)=r^N (  u'^2 +2F_1(u)) +(N-2) r^{(N-1)}u'u)$$
$$E'(r)=r^{(N-1)} (  2N F_1(u) +(N-2) f(u) u)$$
where $u$ is the solution to \eqref{ivp} with  initial condition  $u(0)=3$ and $u'(0)=0$.
It's easy to check that $( 2N F_1(u) +(N-2) f(u) u)< 0 $ if $ u < 3 $, thus $E(r)< 0$ for $r>0$.  If $u$ changes sign, in the point it reaches $0$ we have $E(r)=r^N u'^2 \geq 0$, a contradiction. Thus $3\in {\mathcal P}$ and by Proposition \ref{open-sets} $ (ii)$, $\alpha_*>3$.

Our second step is some observations on the solutions of

\begin{eqnarray}\label{q}
\begin{gathered}
u''+\frac{N-1}{r}u'+A^2 u^q=0,\quad r>0,\\
\end{gathered}
\end{eqnarray}
with  $q>\frac{N+2}{N-2}$, see Miyamoto and Naito in  \cite{miya-naito1,miya-naito2}.

The singular solution of this equation, $v_A$,  i.e. the classical solution in $(0,\infty)$ such that $\lim\limits_{r\to 0} v_A(r)=\infty$, is (see Serrin and Zou \cite[Proposition 3.1]{serrin-zou}),
$$v_A(r)=C(N,q) A{^\frac{-2}{q-1}}r{^\frac{-2}{q-1}},\quad \mbox{where} \quad  C(N, q)= \Bigl(\frac{2}{q-1}\Bigl(N-2-\frac{2}{q-1}\Bigl)\Bigl)^{1/(q-1)},$$
hence $r_A = (C(N,q)){^\frac{(q-1)}{2}} A^{-1} (v_A){^\frac{-(q-1)}{2}}$.
If $u_A$ is a solution of \eqref{q} with initial conditions $u_A(0) =\alpha$ and $u_A'(0) = 0$  then given $\delta>0$ and $0<r_0<r_1$ there is  $\alpha_0$ such that if $\alpha\geq \alpha_0$, and $r_0\leq r\leq r_1$  then $|u_A(r)-v_A(r)|<\delta$ and $|u'_A(r)-v'_A(r)|<\delta$ . From  Miyamoto - Naito we obtain this
for $r_0$ and $r_1$ sufficiently small, from continuity of the solutions we can extend this to compact sets. We will use this result for $r\in [r_0^v, r_1^v]$, where $r_0^v$ and $r_1^v$ are such that $v_A(r_0^v)=\alpha_*+10$ and
 $v_A(r_1^v)=\alpha_*-1$, so that it includes all the values of $r$ we are using.

We now return to our example. For simplicity we will take $q=5$, in this case $C(4,5)=1/2$.
Let  $\epsilon_1 < min\{1/10,  2{^{4}}(\alpha_*+1)^{-5}\}$ and $r_\epsilon$ the radius where  $u_A(r_\epsilon)=\alpha_*+2\epsilon_1$. We will now use  Lemma \ref{epsilon} on the interval $[\alpha_*,\alpha_*+2\epsilon_1]$.

Choose now $0<\delta< 2\epsilon_1/3 $ and $\alpha$ big enough that $|v_A( r_\epsilon)-(\alpha_*+2\epsilon_1)|<\delta$ and $|r_\epsilon u'_A(r_\epsilon)-r_\epsilon v'_A( r_\epsilon)|<\delta$. As $|r_\epsilon v'_A( r_\epsilon) |=\frac{2}{q-1}v_A( r_\epsilon)= \frac{1}{2}v_A( r_\epsilon)$, we have
$$\frac{\alpha_*}{2}< \frac{1}{2}(\alpha_*+2\epsilon_1-\delta) -\delta < r_\epsilon |u'_A(r_\epsilon)|<  \frac{1}{2} (\alpha_*+2\epsilon_1+\delta)+\delta < \frac{\alpha_*}{2} + 2\epsilon_1.$$

 On the other hand, assuming  $A\geq (\alpha_*+2\epsilon)^{2-q}=(\alpha_*+2\epsilon)^{-3}$, we have that  $\epsilon_1 ||f||_+ < 2^{4} A^2$ when $f$ is restricted to $[\alpha_*, \alpha_*+2\epsilon_1]$ and $r_\epsilon^2< 2^{-4}A^{-2} (u_A(r_\epsilon)-\delta)^{-4}= 2^{-4}A^{-2} (\alpha_*+2\epsilon-\delta)^{-4}$.
$$\frac{2 r_\epsilon^2||f||_+}{(N-2)a}\ \epsilon< \frac{1}{a (\alpha_*+2\epsilon-\delta)^{4}}.$$

As $\alpha_*>3$ and $\epsilon<1/10$ we can choose $a = 1/4$ then $2(N-2)a=1<r_\epsilon |u'_A( r_\epsilon)|$ and $2\epsilon <a$,
so by  Lemma \ref{epsilon} we have

  $r_*<2 r_\epsilon < 2(2)^{-2}A^{-1}(\alpha_*+2\epsilon-\delta)^{-2}=\frac{C}{A}$ for a constant $C$  independent of $A$ and $$(N-2)a\le r_*|u_A'(r_*)|\le r_\epsilon|u_A'(r_\epsilon)|+\frac{2r_\epsilon^2||f||_+ }{(N-2)a}\ \epsilon$$
 $$<  \frac{\alpha_*}{2} + 2\epsilon_1 +\frac{4}{3^4} <\frac{\alpha_*}{2} +  \frac{3}{10} $$
  where  $r_*$ is the radius where $u(r_*)=\alpha_*$.

We will finally use Proposition \ref{EnP-f1}.
 We have that for any $A$ if we denote by $\bar a = (N-2)a$ and $ \bar b =\frac{\alpha_*}{2} + \frac{3}{10} $, then
$$\bar a < r_*|u'(r_*)|<  \bar b, $$
and
$$ \bar b =\frac{\alpha_*}{2} +  \frac{3}{10} <(\alpha_*-b) (N-2) =2(\alpha_*-1), $$
We also had $r_*<\frac{C}{A}$.  So , by Proposition   \ref{EnP-f1}, if $A$ is  sufficiently large, then $\alpha \in {\mathcal P}$ for large enough  $\alpha$.
\end{proof}

\subsection*{Example 4}
We finish with a computer simulation for a special case, when $N=4$, $f_1=u^2-u$, $f_i=u^2$ for $i=2,\dots,5$, $\epsilon_i=\frac{1}{10}$, $A_2^2=A_4^2=10$, $A_3^2=A_5^2=\frac{1}{10}$ , (see  Figure \ref{EjA5}). We can see the different behavior of the solutions with $\alpha_i$ with even or odd $i$.

\begin{figure}[h]
  \includegraphics[scale=1]{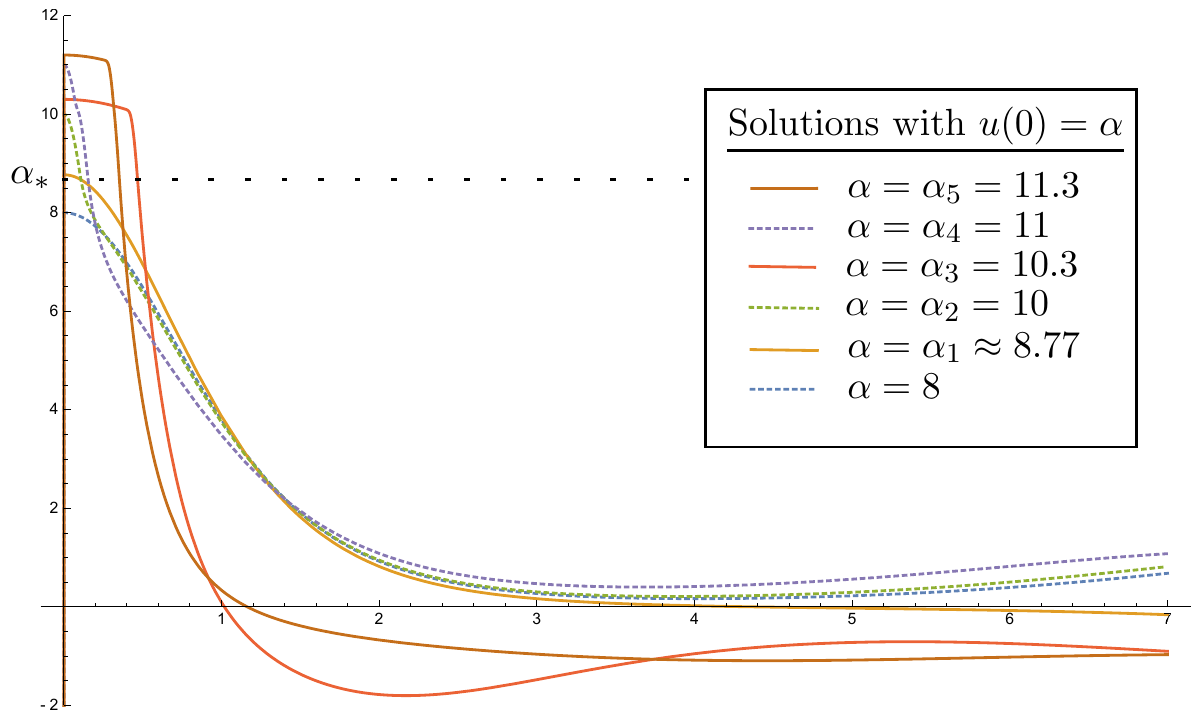}
  \caption{Solutions of Example $4$.}\label{EjA5}
\end{figure}

\section{Appendix}\label{secA}

Let  $u(r)$ be a solution of
\begin{eqnarray}\label{ap}
\begin{gathered}
u''+\frac{N-1}{r}u'+f(u)=0,\quad r>0,\quad N> 2,\\
u(0)=\alpha,\quad u'(0)=0
\end{gathered}
\end{eqnarray}
and let $(H_1)$,$(H_2 )$, $(H_3 )$ and $(H_4)$ be as in the Introduction, with $f$ instead of $f_1$ and $F(s)$ instead of $F_1(s)$.

There are some known facts about the solutions of this equation that we have not found proven with our conditions, so, for the sake of completeness, we give  a proof of them in this appendix.  The main one is the following theorem.

\begin{theorem}\label{unic}
 If  assumptions  $(H_1)$,$(H_2 )$ and $(H_4)$ are satisfied then
the ground state is unique and
${\mathcal P}$=$(b,\alpha_*)$ and ${\mathcal N}$ =$(\alpha_*, \infty)$.

\end{theorem}
In order to prove this theorem we need the following propositions.

Let $u_1$, $u_2$ be two solutions of \eqref{ap},  such that $u_1(0)=\alpha_1< u_2(0)=\alpha_2$.

\begin{proposition}\label{SI2}
Let f satisfy $(H_1)$-$(H_2 )$ and  let  the solution with initial condition
 $\alpha_*$  be a ground state. Then there exists $\delta >0$ such that if $\alpha_1,\alpha_2 \in (\alpha_*-\delta,\alpha_*+\delta)$,
 then $u_1$ and $u_2$ a point $ u_1(r_I)=u_2(r_I) > 0 $.
\end{proposition}
\begin{proof}
See, for example, \cite{cghy2}.

 \end{proof}

We will use the functional introduced first by Erbe and Tang in \cite{et}, and  this proof uses many of their ideas.

Let $u_1$, $u_2$ be two solutions of \eqref{ap},  such that $u_1(0)=\alpha_1< u_2(0)=\alpha_2$.
For each of these solutions  $u_i$,  let

$$P_i(s)=-2N\frac{F}{f}(s)\frac{r_i^{N-1}(s)}{r_i'(s)}-\frac{r_i^N(s)}{(r_i'(s))^2}
-2r_i^N(s)F(s),\quad $$
where $r_i(s)$ denotes the inverse of $u$ in the interval $(0,R(\alpha))$, and note that

\begin{equation}\frac{d P_i}{d s}(s)=\left(N-2-2N\Bigl(\frac{F}{f}\Bigr)'(s)\right)\frac{r_i^{N-1}(s)}{r_i'(s)}.
\end{equation}

 In what follows  $u_1(r_I)=u_2(r_I)=s_I$ will be the largest point of intersection between  the functions $u_1(r)$ and $u_2(r)$.

\begin{proposition}\label{Inter}
Let $f$ satisfy $(H_1)$ and $(H_2 )$.  Then
\begin{enumerate}
\item[(i)]$\frac{r_2^{N-1}(s)}{r_2'(s)}<\frac{r_1^{N-1}(s)}{r_1'(s)}  $ for $s$ in $[s_I, \alpha_1]$
\item[(ii)]$P_2(s_I)<P_1(s_I)$
\end{enumerate}
\end{proposition}

\begin{proof}
To prove $(i)$ compute the derivative of $\frac{r_1^{N-1}(s)r_2'(s)}{r_2^{N-1}(s)r_1'(s)} $, see Serrin and Tang, \cite{st}. For $(ii)$ use  $(i)$ in the expression for $P_s$ above.

\end{proof}

\begin{proposition}\label{finalI}

Let f satisfy $(H_1)$ and $(H_2 )$. Let $u_1$, $u_2$ be two solutions of \eqref{ivp},  such that $u_1(r_I)= u_2(r_I)= s_I$,  $u'_1(r_I)= \bar \alpha_1$, $u'_2(r_I)= \bar \alpha_2$ with $ \bar \alpha_2 <\bar \alpha_1<0 $.
If

\noindent i)   $s_I >\beta$  and  $ P_2(s_I)< P_1(s_I)< 0 $ , where $P_i=P(s,u_i)$,\\
or\\
\noindent ii) $s_I \leq \beta$ ,\\
then $r_2(s)<r_1(s)$  and there is a constant $c >0$ such that $ r_1(s)|u'(r_1(s))|< r_2(s) |u'(r_2(s))|-c $ in the interval $ (a, s_I)$  where they are both
defined.

\end{proposition}

\begin{proof}
We will first study the case $s_I > \beta$ .

First suppose there is  a largest $s_0>\beta $ such that $ r_1(s)>r_2(s )$ for $[s_0, s_I)$ and $\frac{
r_1}{ r_1'}(s_0)=\frac{ r_2}{  r_2'}(s_0)$ . Let
$$A=  \frac{
 r^{N-1}_2 r_1'}{ r^{N-1}_1 r_2'} (s_0)= \frac{ r^{N-2}_2}{ r^{N-2}_1} (s_0)<1.$$
Observe that $$ \left(\frac{ r^{N-1}_2 r_1'}{ r_1^{N-1} r_2'} \right)_s = f(u) \frac{ r_2^{N-1}}{ r_1^{N-1}}\frac{ r_1'(s)}{ r_2'(s)} \big(( r_1')^2-( r_2')^2\big)>0\quad\mbox{in }[s_0,s_I]$$
hence
$$\frac{ r^{N-1}_2 r_1'}{ r_1^{N-1} r_2'}\ge A.$$
As $ P_2(s_I)< P_1(s_I)< 0 $ ,  we have $0<( P_1-\ P_2)(s_I)<(A P_1- P_2)(s_I).$
By direct computation we have
$$\frac{d}{ds}\left(A P_1- P_2\right)=\left(N-2-2N\Bigl(\frac{F}{f}\Bigr)'(s)\right)\left(A\frac{\ r^{N-1}_1}{ \ r_1'}-\frac{ r^{N-1}_2}{  r_1'}\right)\leq 0 $$
in $[s_0,s_I]$ and thus $(A P_1-\bar P_2)(s_0)>0$.
On the other hand,
$$(A P_1- P_2)(s_0)=-2F(s_0) r_2^{N-2}( r_1^2- r_2^2)<0,$$
a contradiction, therefore
$ r_2(\beta)<r_1(\beta )$  and $\frac{r_2}{ r_2'}(\beta)<\frac{ r_1}{  r_1'}(\beta)$.

Now we will use the following functional
$$W_i(s)= r_i(s)\sqrt{(u_i'( r_i(s)))^2+2F(s)},\quad i=1,2.$$
found  in Franchi, Lanconelli  and Serrin,  \cite{fls} or Serrin and Tang  \cite{st}.
Observe that, as $\frac{r_i}{ r_i'}(s) =r_i(s)(u_i'( r_i(s))$, $ W_1(\beta)<W_2(\beta)$ and
\ben
\frac{d}{ds}(W_1-W_2)(s)=(N-2)\left(\frac{|u_1'( r_1(s))|}{\sqrt{(u_1'(r_1(s)))^2+2F(s)}}-\frac{|u_2'(r_2(s))|}{\sqrt{(u_2'(r_2(s)))^2+2F(s)}}\right)\\
\qquad\quad-2F(s)\left(\frac{1}{|u_1'( r_1(s))|\sqrt{(u_1'(r_1(s)))^2+2F(s)}}-\frac{1}{|u_2'( r_2(s))|\sqrt{(u_2'( r_2(s)))^2+2F(s)}}\right),
\een
For $s<\beta$, so that $F(s)<0$,
the functions
$$x\mapsto \frac{x}{\sqrt{x^2+2F(s)}}\quad \mbox{and}\quad  x\mapsto \frac{1}{x\sqrt{x^2+2F(s)}}$$
are decreasing for $x>0$. Hence  we obtain that as long as $|u_1'(r_1(s))|\le|u_2'(r_2(s))|$,  $W_1-W_2$ increases and  also $ r_2(s)< r_1(s)$.   If $|u_1'( r_1(s_1))|=|u_2'( r_2(s_1))|$ for some $s_1\in(\bar s,\beta)$, with $|u_1'(r_1(s))|<|u_2'( r_2(s))|$ in $(s_1,\beta)$, then $W_1(s_1)<W_2(s_1)$ and hence $ r_1(s_1)<r_2(s_1)$, a contradiction.
 So $|u_1'(r_1(s))|\le|u_2'(r_2(s))|$ and $ r_2(s)< r_1(s)$ for all $s$ for all $s\in[\bar s,\beta]$, and the result follows.

For the case  $s_I \leq \beta$
we have  $ r_1(s_I ) =r_2(s_I ) $  and  $|u_1'( r_1(s_I ))|<|u_2'( r_2(s_I ) )|$, so we can repeat the argument with the functional $W$, starting at  $(s_I ) $ instead of at $\beta$, finishing the proof.

\end{proof}

\begin{proof}[Proof of Theorem \ref{unic}]

Let $\alpha_*$ be the initial condition of a ground state, $ u_*(r)$, then $\lim_{r\to\infty} r |u_*'(r)| = 0$. This is because  the function $r^{2(N-1)}(u_*'(r))^2+2F(u(r))$ is positive and decreasing for $F$ negative, so has a finite limit. Dividing by $r^{2(N-2)}$ we obtain the result.

Choose now  $\delta >0$ as in Proposition \ref{SI2}  and  $\alpha_1,\alpha_2 \in (\alpha_*-\delta,\alpha_*+\delta)$. By Proposition \ref{finalI}
 $r_2(s)<r_1(s)$  and there is a constant $c >0$ such that  $r_1(s)|u_1'(r_1(s))|< r_2(s) |u_2'(r_2(s))|-c$  in the interval $ (a, s_I)$.
If $ \alpha_1=\alpha_*$, then  $u_2(r_2(s)) $ is  defined in $ (0, s_I)$  and as  $0=r_1(0)|u_1'(r_1(0))|$ we have  $0<r_2(0)|u_2'(r_2(0))|$.
So $\alpha_2$ is in $\mathcal N$. On the other hand, if  $ \alpha_2=\alpha_*$, and $ \alpha_1 $ is defined in $ (0, s_I)$ then $0 >r_1(0)|u_1'(r_1(0))|$,
a contradiction, so in this case  $\alpha_1$ is in $\mathcal P$.
We have proven that given $\alpha_* \in \mathcal G $ there is $\delta >0$ such that $(\alpha_*-\delta,\alpha_*)\subset\mathcal P$ and $(\alpha_*,\alpha_*+\delta)\subset\mathcal N$.
 As $\mathcal P$ and $\mathcal N$ are open the result follows.

 \end{proof}

If we assume $(H_1)$,$(H_2 )$ and $(H_4 )$  then any two solutions of \eqref{ap} with initial conditions $\alpha_i>b$ do not necessarily intersect.  They  do if   you add the hypothesis  $(H_3)$, i. e. that $f$ is super-linear for $b<s$.

\begin{proposition}\label{SI1}
Let f satisfy $(H_1)$-$(H_3 )$ and $u_1$, $u_2$ be two solutions of \eqref{ap} with initial conditions $\alpha_i>b$,  then the two solutions intersect in a point $u_1(r_I)=u_2(r_I) \geq b$.
\end{proposition}

\begin{proof}

Assume $b<\alpha_1<\alpha_2$ and $u_1<u_2$ for all $s\in(b,\alpha_1)$. By multiplying the equations satisfied by $u_1$, $u_2$ by $u_2-b$ and
 $u_1-b$ respectively and subtracting, we obtain
\begin{eqnarray}\label{dif0}
  \frac{d}{dr} \Bigl(r^{n-1}(u_1'(u_2-b)-u_2'(u_1-b))\Bigr)
  = -r^{n-1}(u_1-b)(u_2-b)\left(\frac{f(u_1)}{u_1-b}
    -\frac{f(u_2)}{u_2-b}\right).
\end{eqnarray}
 Since by $(H_3)$ the right hand side of (\ref{dif0}) is nonnegative in
 $[0,r_1(b)]$, by integrating \eqref{dif0} over this interval,
 we obtain a contradiction.

 \end{proof}

We will use this result to obtain  the behaviour of solutions in $\mathcal P$.

 \begin{proposition}\label{P}

Let f satisfy $(H_1)$-$(H_4 )$ and $u_1$, $u_2$ be two solutions of \eqref{ivp},  such that $u_1(0)= \alpha_1$ $u_2(0)= \alpha_2$   $u'_1(0)= u'_2(0)=0$,  with    $ b<\alpha_1 <\alpha_2< \alpha_*$. Then $u_1$ and $u_2$ intersect each other once and only once in $(0,\min\{R(\alpha_1),R(\alpha_2)\})$. Such intersection occurs at a point  $u_1=u_2>b$.

\end{proposition}

\begin{proof}
We will give a sketch of the proof.
From the previous proposition, we have that the solutions intersect at least once with value greater than $b$.
Use now the Erbe-Tang functional $P$ defined above from  $\alpha_i$ to the point of intersection, and observe that the assumptions of  Proposition \ref{final1} are satisfied, hence they do not intersect again.

\end{proof}

\begin{proposition}\label{SI1}
Let f satisfy $(H_1)$-$(H_4 )$ and $u_1$, $u_2$ be two solutions of \eqref{ap}, then the two solutions intersect in a point $u_1(r_I)=u_2(r_I) \geq b$.
\end{proposition}

\begin{proof}

Assume $\alpha_1<\alpha_2$ and $u_2>u_1$ for all $s\in(b,\alpha_1)$. By multiplying the equations satisfied by $u_1$, $u_2$ by $u_2-b$ and
 $u_1-b$ respectively and subtracting, we obtain

 \begin{equation}\label{dif}
  \frac{d}{dr} \Bigl(r^{n-1}(u_1'(u_2-b)-u_2'(u_1-b))\Bigr)
  = -r^{n-1}(u_1-b)(u_2-b)\left(\frac{f(u_1)}{u_1-b}
    -\frac{f(u_2)}{u_2-b}\right).
 \end{equation}
 Since by $(H_4)$ the right hand side of \eqref{dif} is nonnegative in
 $[0,r_1(b)]$, by integrating \eqref{dif} over this interval,
 we obtain a contradiction.

 \end{proof}

On the other hand, if we do not assume $ H_3 $, we have the following result.

 \begin{proposition}\label{P}

Let f satisfy $(H_1)$-$(H_4 )$ and $u_1$, $u_2$ be two solutions of \eqref{ivp},  such that $u_1(0)= \alpha_1$ $u_2(0)= \alpha_2$   $u'_1(0)= u'_2(0)=0$,  with    $ b<\alpha_1,  \alpha_2< \alpha_*$. Then $u_1$ and $u_2$ intersect each other once and only once in $(0,\min\{R(\alpha_1),R(\alpha_2)\})$. Such intersection occurs at a point  $u_1=u_2>b$.

\end{proposition}

\begin{proof}
We will give a sketch of the proof.

Property $(H_3)$ implies that $f_1$ is superlinear for $s>b$, hence there is at least one intersection. Indeed, assume $\alpha_1<\alpha_2$ and $u_2>u_1$ for all $s\in(b,\alpha_1$. By multiplying the equations satisfied by $u_1$, $u_2$ by $u_2-b$ and
 $u_1-b$ respectively and subtracting, we obtain
 \begin{equation}\label{dif}
 \frac{d}{dr} \Bigl(r^{n-1}(u_1'(u_2-b)-u_2'(u_1-b))\Bigr)
  = -r^{n-1}(u_1-b)(u_2-b)\left(\frac{f(u_1)}{u_1-b}
    -\frac{f(u_2)}{u_2-b}\right).
 \end{equation}
 Since by $(H_4)$ the right hand side of \eqref{dif} is nonnegative in
 $[0,r_1(b)]$, by integrating \eqref{dif} over this interval,
 we obtain a contradiction.

Use the Erbe-Tang functional $P$ defined above from  $\alpha_i$ to the point of intersection, and observe that the assumptions of  Proposition \ref{final1} are satisfied, hence they do not intersect again.

\end{proof}



\begin{thebibliography}{AAAAA}
	
\bibitem[AT1]{at1} {\sc Adachi, S. Tanaka, K.}, Four positive solutions for the semilinear elliptic equation:  $-\Delta u+u=a(x)u^p+f(x)$ in $R^N$. {\em Calc. Var. Partial Differential Equations} {\bf 11 } (2000), no. 1, 63--95.
\bibitem[AT2]{at2}  {\sc Adachi, S. Tanaka, K.}, Existence of positive solutions for a class of nonhomogeneous elliptic equations in $R^N$. {\em Nonlinear Anal. Ser. A: Theory Methods,} {\bf 48 }(2002), no. 5,  685--705.

\bibitem[AW]{aw}  {\sc  Ao, W.,    Wei, J.}, Infinitely many positive solutions for nonlinear equations with non-symmetric potential,  {\em Calc. Var. Partial Differential Equations} {\bf 51 }(2014), no. 3--4, 761--798.

\bibitem[CZ]{cao}  {\sc Cao, D-M., Zhou, H.,} Multiple positive solutions of nonhomogeneous semilinear elliptic equations in $R^N$. {\em Proc. Roy. Soc. Edinburgh Sect. A} {\bf 126 }(1996), no. 2, 443--463.
\bibitem[CK]{cku} {\sc Castro, A., Kurepa, A., }	Infinitely many radially symmetric solutions to a superlinear Dirichlet problem in a ball, {\em Proc. Amer. Math. Soc.}, {\bf 101 }(1987), 57--64.
\bibitem[CM]{cm19}{\sc Cerami, G., Molle, R.}, Infinitely many positive standing waves for Schrödinger equations with competing coefficients. {\em Comm. Partial Differential Equations} {\bf 44 }(2019), no. 2, 73--109.
\bibitem[CPS]{cps1} {\sc Cerami, G., Passaseo, D., Solimini, S.}, Infinitely many positive solutions to some scalar field equations with nonsymmetric coefficients. Comm. Pure Appl. Math. 66 (2013), no. 3, 372--413.

\bibitem[CMP]{cmp}{\sc Cerami, G., Molle, R., Passaseo, D.}, Multiplicity of positive and nodal solutions for scalar field equations.
{\em J. Differential Equations} {\bf 257 }(2014), no. 10, 3554--3606.
\bibitem[CEF1]{cfe1}{\sc Cort\'azar, C., Felmer, P., Elgueta, M., }On a semilinear elliptic problem in $\RR^N$ with a
	non Lipschitzian nonlinearity, {\em Advances in Differential Equations} {\bf 1 }(1996), 199-218.

	
\bibitem[CEF2]{cfe2}{\sc Cort\'azar, C., Felmer, P., Elgueta, M., }Uniqueness of positive  solutions of $\Delta u+f(u)=0$
	in $\RR^N$, $N\ge 3$, {\em Archive Rat. Mech. Anal.} {\bf 142 }(1998), 127-141.
	

	
\bibitem[CGHY]{cghy2}{\sc Cort\'azar, C., Garc\'\i a-Huidobro, M., Yarur, C. }On the uniqueness of sign changing bound state solutions of a semilinear equation. {\em Ann. Inst. H. Poincar\'e Anal. Non Lin\'eaire} {\bf 28 }(2011), no. 4, 599--621.
	
\bibitem[CGHH]{cghh2}{\sc Cort\'azar, C., Garc\'\i a-Huidobro, M., Herreros, P. } Multiplicity results for sign changing bound state solutions of a semilinear equation. {\em J. Differential Equations} {\bf 259 }(2015), no. 12, 7108--7134.
\bibitem[DDG]{ddg}{\sc D\'avila, J., Del Pino, M., Guerra, I.} Non-uniqueness of positive ground states of non-linear Schr\"odinger equations.  {\em Proc. London Math. Soc}{\bf 106}(2013), no. 3, 318--344.


\bibitem[DWY]{dwy}  {\sc Del Pino, M., Wei,  J.,   Yao, W.} Intermediate reduction method and infinitely many positive solutions of nonlinear Schrödinger equations ,  {\em Calc. Var. Partial Differential Equations} {\bf 53 }(2015), no. 1--2, 473--523. 	

\bibitem[ET]{et} {\sc  Erbe, L.,   Tang, M., } Uniqueness theorems for
positive solutions of quasilinear elliptic equations in a ball, {\it
	J. Diff. Equations} {\bf 138} (1997), 351-379.

\bibitem[FG]{fg}{\sc Ferrero, A., Gazzola, F.} On  subcriticality assumptions for the existence of ground states of quasilinear elliptic equations, {\em Advances in Diff. Equat.}, {\bf 8 }(2003), no 9, 1081--1106.
	
	
\bibitem[FLS]{fls}{\sc Franchi, B., Lanconelli, E., Serrin, J., }Existence and Uniqueness of nonnegative
	solutions of quasilinear equations in $\RR^n$, {\em Advances in mathematics} {\bf 118 }(1996), 177-243.
	
\bibitem[GST]{gst}{\sc Gazzola, F., Serrin, J. and Tang, M., }Existence of ground states and free boundary value problems for quasilinear elliptic operators. {\em Advances in Diff. Equat.} {\bf 5 }(2000), no. 1-3, 1-30.
	
\bibitem[HL]{hl}  {\sc Hsu, T., Lin, H.,} Four positive solutions of semilinear elliptic equations involving concave and convex nonlinearities in $R^N$. {\em J. Math. Anal. Appl.} {\bf 365 }(2010), no. 2, 758--775.
	
\bibitem[MN1]{miya-naito1}  {\sc Miyamoto, Yasuhito; Naito, Yūki}. Singular extremal solutions for supercritical elliptic equations in a ball. {\em J. Differential Equations} {\bf 265 }(2018), no. 7, 2842--2885.

\bibitem[MN2]{miya-naito2}{\sc Miyamoto, Yasuhito; Naito, Yūki}. Fundamental properties and asymptotic shapes of the singular and classical radial solutions for supercritical semilinear elliptic equations. {\em NoDEA Nonlinear Differential Equations Appl.} {\bf 27 }(2020), no. 6, Paper No. 52, 25 pp.	
	
\bibitem[McLS]{ms} {\sc  McLeod, K.,  Serrin, J., }
	Uniqueness of positive radial solutions
	of $\Delta u+f(u)=0$ in $\RR^N,$   {\em Arch. Rational Mech. Anal.},
	{\bf 99} (1987), 115-145.
	
\bibitem[MP21]{mp21}{\sc Molle, R.; Passaseo, D.,} Infinitely many positive solutions of nonlinear Schrödinger equations. {\em Calc. Var. Partial Differential Equations} {\bf 60 }(2021), no. 2, Paper No. 79, 35 pp.
	
\bibitem[NS1]{ni-serrin} {\sc Ni, W.-M., Serrin, J.} Nonexistence theorems for singular solutions of quasilinear
partial differential equations. {\em Commun. Pure Appl. Math.} {\bf 39 }(1986), 379--399.	

\bibitem[PeS1]{pel-ser1}{\sc Peletier, L., Serrin, J., }Uniqueness of positive
solutions of quasilinear equations, {\em Archive Rat. Mech. Anal.} {\bf 81 }(1983), 181-197.

\bibitem[PeS2]{pel-ser2}{\sc Peletier, L., Serrin, J., }Uniqueness of nonnegative
solutions of quasilinear equations, {\em J. Diff. Equat.} {\bf 61 }(1986), 380-397.
	
\bibitem[PS]{pu-ser}{\sc Pucci, P.,  R., Serrin, J., }Uniqueness of ground states for quasilinear elliptic operators, {\em
		Indiana Univ. Math. J.} {\bf 47 }(1998), 529-539.
	
\bibitem[ST]{st} {\sc  Serrin, J., and  Tang, M., }
	Uniqueness of ground states for quasilinear elliptic equations,
	{\em Indiana Univ. Math. J.} {\bf 49} (2000),  897-923.

\bibitem[SZ]{serrin-zou}{\sc Serrin, J., Zou, H.} Classification of positive solutions of quasilinear elliptic equations.
{\em Topol. Methods Nonlinear Anal.} {\bf 3 }(1994), 1–25.	

 \bibitem[WW]{wei-wu} {\sc Wei, J.; Wu, Y.} Normalized solutions for Schrödinger equations with critical Sobolev exponent and mixed nonlinearities. {\em J. Funct. Anal.} {\bf 283 }(2022), no. 6, Paper No. 109574, 46 pp.

 \bibitem[WY]{wy}  {\sc Wei, J., Yan, S.}, Infinitely many positive solutions for the nonlinear Schr\"odinger equations in $\mathbb R^N$. {\em Calc. Var. Partial Differential Equations} {\bf 37 }(2010), no. 3-4, 423--439.
\end{thebibliography}
\end{document}